\documentclass[11pt,a4paper]{article}

\usepackage[left=2.3cm, top=2.3cm,bottom=2.3cm,right=2.3cm]{geometry}

\usepackage{mathtools,amssymb,amsthm,mathrsfs,calc,graphicx,stmaryrd,xcolor,cleveref,dsfont,tikz,pgfplots,bbm,float}
\usepackage[british]{babel}
\usepackage{amsfonts}              

\numberwithin{equation}{section}
\numberwithin{figure}{section}

\newtheorem {theorem}{Theorem}[section]
\newtheorem {proposition}[theorem]{Proposition}
\newtheorem {lemma}[theorem]{Lemma}
\newtheorem {corollary}[theorem]{Corollary}

{\theoremstyle{definition}

\newtheorem {remark}[theorem]{Remark}

}
{
\theoremstyle{theorem}
}

\newcommand{\Vol}{\operatorname{Vol}}

\def\ba{\begin{array}}
\def\ea{\end{array}}
\def\bea{\begin{eqnarray} \label}
\def\eea{\end{eqnarray}}
\def\be{\begin{equation} \label}
\def\ee{\end{equation}}
\def\bit{\begin{itemize}}
\def\eit{\end{itemize}}
\def\ben{\begin{enumerate}}
\def\een{\end{enumerate}}

\def\BB{\mathbb{B}}

\def\EE{\mathbb{E}}

\def\NN{\mathbb{N}}
\def\PP{\mathbb{P}}

\def\RR{\mathbb{R}}
\def\RRd1{\mathbb{R}^{d+1}}
\def\SS{\mathbb{S}}
\def\SSd{\mathbb{S}^d}

\def\b{\beta}
\def\g{\gamma}

\def\k{\kappa}
\def\l{\lambda}

\def\o{\omega}
\def\G{\Gamma}

\def\mw{\mathscr{W}}

\def\dint{\textup{d}}

\newcommand{\eee}{{\rm e}}
\newcommand{\tobeta}{\overset{}{\underset{\beta\to+\infty}\longrightarrow}}
\newcommand{\ind}{\mathbbm{1}}

\begin{document}

\title{\bfseries Expected intrinsic volumes and facet numbers\\ of random beta-polytopes}

\author{Zakhar Kabluchko, Daniel Temesvari and Christoph Th\"ale}

\date{}

\maketitle

\begin{abstract}
Let $X_1,\ldots,X_n$ be i.i.d.\ random points in the $d$-dimensional Euclidean space sampled according to one of the following probability densities:
$$
f_{d,\beta} (x) = \text{const} \cdot (1-\|x\|^2)^{\beta}, \quad \|x\| < 1, \quad \text{(the beta case)}
$$
and
$$
\tilde f_{d,\beta} (x) = \text{const} \cdot (1+\|x\|^2)^{-\beta}, \quad x\in\RR^d, \quad \text{(the beta' case).}
$$
We compute exactly the expected intrinsic volumes and  the expected number of facets of the convex hull of $X_1,\ldots,X_n$. Asymptotic formulae were obtained previously by Affentranger [The convex hull of random points with spherically symmetric distributions, 1991].
By studying the limits of the beta case when $\beta\downarrow -1$, respectively $\beta \uparrow +\infty$, we can also cover the models in which  $X_1,\ldots,X_n$ are uniformly distributed on the unit sphere or normally distributed, respectively.
 We obtain similar results for the random polytopes defined as the convex hulls of $\pm X_1,\ldots,\pm X_n$ and $0,X_1,\ldots,X_n$. One of the main tools used in the proofs is the Blaschke-Petkantschin formula.
\noindent
\bigskip
\\
{\bf Keywords}. Beta distribution, beta' distribution, Blaschke-Petkantschin formula, convex hull, expected volume, intrinsic volumes, number of facets, random polytope, symmetric convex hull.\\
{\bf MSC 2010}. 	52A22, 52B11, 53C65, 60D05.
\end{abstract}

\tableofcontents

\section{Introduction}
This paper is concerned with a particular class of random polytopes in the $d$-dimensional Euclidean space $\RR^d$.
Let $X_1,\ldots,X_n$ be independent and identically distributed (i.i.d.)\ points in $\RR^d$ chosen according to one of the following probability densities:
 \begin{itemize}
 	\item[(a)] the beta-distribution: $f_{d,\b}(x)=c_{d,b} \left( 1-\left\| x \right\|^2 \right)^\b$, $\|x\| <  1$, where $\b>-1$,
 	\item[(b)] the beta$^\prime$-distribution: $\tilde{f}_{d,\b}(x)=\tilde{c}_{d,\b} \left( 1+\left\| x \right\|^2 \right)^{-\b}$, $x\in\RR^d$, where $\b>\frac{d}{2}$.
 \end{itemize}
The normalizing constants $c_{d,\b}$ and $\tilde{c}_{d,\b}$ of the densities $f_{d,\b}$ and $\tilde{f}_{d,\b}$, respectively, are given by
\begin{equation}\label{eq:c_d_beta}
c_{d,\b}= \frac{ \G\left( \frac{d}{2} + \b + 1 \right) }{ \pi^{ \frac{d}{2} } \G\left( \b+1 \right) }
\;\;\; \text{and} \;\;\;
\tilde{c}_{d,\b}= \frac{ \G\left( \b \right) }{\pi^{ \frac{d}{2} } \G\left( \b - \frac{d}{2} \right) }.
\end{equation}
Moreover, $\|x\| = (x_1^2+\ldots+x_d^2)^{1/2}$ denotes the Euclidean norm of the vector $x= (x_1,\ldots,x_d)\in\RR^d$.

We will be interested in various kinds of random polytopes generated by these points. Let us denote in this paper by $[a_1,\ldots,a_n]$ the convex hull of the points $a_1,\ldots, a_n\in\RR^d$. Assuming that $X_1,\ldots,X_n$ have the beta distribution with parameter $\beta$, we define the following random polytopes:
$$
P_{n,d}^\beta := [X_1,\ldots,X_n],
\qquad
S_{n,d}^\beta := [\pm X_1,\ldots,\pm X_n],
\qquad
Q_{n,d}^\beta := [0,X_1,\ldots,X_n].
$$
Analogously, $\tilde{P}_{n,d}^\beta$, $\tilde{S}_{n,d}^\beta$ and $\tilde{Q}_{n,d}^\beta$ denote the similarly generated polytopes from $n$ independent and identically beta$^\prime$-distributed points in $\RR^d$ with parameter $\b$.

The main aim of the present work is to compute {\it exactly} the expected  volumes of the random polytopes $P_{n,d}^\beta, S_{n,d}^\beta, Q_{n,d}^\beta, \tilde{P}_{n,d}^\beta,\tilde{S}_{n,d}^\beta, \tilde{Q}_{n,d}^\beta$. Besides, we shall compute the expectations of some other functionals of these polytopes as well.  We denote the intrinsic volumes of a convex body $K\subset \RR^d$ by $V_k(K)$, $k=0,\dots,d$. Note that $V_d(K)=\Vol_d(K)$ is the $d$-dimensional Lebesgue volume, $V_{d-1}(K)$ is half the surface area, $V_1(K)$ is a constant multiple of the mean width and $V_0(K)=1$ is the Euler characteristic of the convex body $K$.  The $k$-dimensional Lebesgue volume of a $k$-dimensional set  $K\subset \RR^d$, where $k\leq d$, will be denoted by $\Vol_k(K)$.
Theorems~\ref{t2}, \ref{t2_prime}, \ref{t1}, \ref{t1_prime} and Proposition~\ref{prop:expected_intrinsic} below state  exact expressions for the expected values of the following functionals of the above polytopes:
\begin{itemize}
\item[(a)] the intrinsic volumes $V_k$ (including the volume, the surface area and the mean width),
\item[(b)] the $T$-functionals (introduced by Wieacker~\cite{Wieacker})
$$
T_{a,b}^{d,k}(P)= \sum_{F \in \mathcal{F}_k(P)} \eta^{a}(F) \Vol_k^b(F),
$$
where $a, b \geq 0$ are parameters, $\mathcal{F}_k(P)$ is the set of all $k$-dimensional faces of a convex polytope $P$ for $k=0,\dots,d$, and  $\eta(F)$ is the Euclidean distance from the affine hull of the $k$-dimensional face $F$ to the origin.
\end{itemize}

There exists a large literature about random polytopes. The rapid development started with the influential papers of Efron \cite{Efron65} and R\'enyi and Sulanke \cite{RenyiSulanke} as well as with the thesis of Wieacker \cite{Wieacker}, where mainly random polytopes $K_n$ generated by $n\geq d+1$ independent and uniformly distributed points in a prescribed convex body $K\subset\RR^d$ have been studied. The particular interest was in the goodness of the approximation of $K$ by $K_n$, which is measured, for example, by the expected intrinsic volume difference $V_k(K)-\EE V_k(K_n)$, $k=1,\ldots,d$. More recently, also higher order moments and central limit theorems for various geometric quantities of random polytopes have been investigated intensively. For selected examples we refer to the works of B\'ar\'any, Fodor and Vigh \cite{BaranyFodorVigh}, B\'ar\'any and Vu \cite{BaranyVu} and Reitzner \cite{ReitznerCombStructure05,ReitznerCLT05}, and to the survey articles of B\'ar\'any \cite{BaranySurvey}, Hug \cite{HugSurvey}, Reitzner \cite{ReitznerSurvey} and Schneider \cite{SchneiderSurvey} as well as the references cited therein.

More closely related to the content of the present paper are the works of Aldous, Fristedt, Griffin and Pruitt \cite{AldousEtAl}, Affentranger \cite{Affentranger88,Affentranger91}, Buchta and M\"uller \cite{BuchtaMueller}, Buchta, M\"uller and Tichy \cite{BuchtaMuellerTichy}, Carnal \cite{Carnal}, Dwyer \cite{Dwyer}, Eddy and Gale \cite{EddyGale} as well as of Raynaud \cite{Raynaud}, where random polytopes generated by random points with spherically symmetric distributions in $\RR^d$ were investigated, especially for the beta- and in some cases also for the beta'-distribution. The asymptotic behavior of the expected number of facets and (several special cases of) the expected $T$-functionals $T_{a,b}^{d,k}$ of the polytopes $P_{n,d}^\beta$ has been determined, as $n\to\infty$ (in some cases restricting to the planar case $d=2$). A similar analysis for the Gaussian random polytope is the content of the works of Affentranger and Schneider~\cite{AffentrangerSchneider92} and Hug, Munsonius and Reitzner \cite{HugMunsoniusReitzner}; see also~\cite{KabluchkoZaporozhets}. It is one of the goals of the present paper to unify these approaches, to extend them to other random polytope models (such as symmetric random polytopes) and to develop systematically in each case {\em explicit} integral formulae for the geometric functionals described above. We shall also discuss the special cases $d=2$ and $d=3$ separately, where in some cases our formulae coincide with results known from the existing literature and in other cases they lead to formulae that were not available before. In this context we especially refer to the early results of Buchta \cite{BuchtaEllipsoid}, Buchta, M\"uller and Tichy \cite{BuchtaMuellerTichy} and that of Efron \cite{Efron65}.

In the following, in all our formulae we implicitly assume that $n\geq d+1$.

\medbreak

The rest of this paper is structured as follows. In Section \ref{sec:MainResults} we present the main results of the present work. Moreover, in Section \ref{sec:SpecialCases} we collect a number of special cases of our results, in particular, for random polytopes generated by random points uniformly distributed in the $d$-dimensional unit ball and on the $(d-1)$-dimensional unit sphere. We also discuss there special formulae for $d=2$ and $d=3$. Some auxiliary results needed in our proofs are collected in Section \ref{sec:AuxResults}, whereas Sections \ref{sec:Proof1} and \ref{sec:proof_exp_volumes} contain all proofs.

\begin{remark}
On our web-pages we provide two Mathematica worksheets (one for each distribution), which allow the reader to compute specific values for
the expected volume, surface area, mean width and facet number of $P_{n,d}^\beta, S_{n,d}^\beta, Q_{n,d}^\beta, \tilde{P}_{n,d}^\beta,\tilde{S}_{n,d}^\beta$ and $\tilde{Q}_{n,d}^\beta$ for any fixed dimension, any given number of points and any parameter $\beta$ of the distribution. They also allow to compute explicitly the expected intrinsic volumes
and expected value of the $T$-functional. Some particular values for dimensions $2$ and $3$ are collected in the appendix.
\end{remark}

\section{Main results}\label{sec:MainResults}

\subsection{Expected volumes and intrinsic volumes}

In the one-dimensional case, the cumulative distribution functions of the probability distributions associated with the densities $f_{1,\b}$ in the beta- and $\tilde{f}_{1,\b}$ in beta'-case, respectively, are denoted by
\begin{align}
F_{1,\b}(h)
&=
c_{1,\b} \int\limits_{-1}^h  \left( 1- x^2 \right)^\b \,\dint x,
\quad h\in [-1,1], \label{eq:def_F_d_beta}\\
\tilde{F}_{1,\b}(h)
&=
\tilde c_{1,\b} \int\limits_{-\infty}^h  \left( 1+ x^2 \right)^{-\b} \,\dint x,
\quad h\in\RR. \label{eq:def_F_d_beta_prime}
\end{align}

For all $k\in \NN$, the $k$-dimensional volume $\kappa_k$ of the $k$-dimensional unit ball $\BB^k$ and the surface area $\o_k$ of the unit sphere $\SS^{k-1} = \partial \BB^k$ are given by
$$
\k_k=\Vol_k(\BB^k) = \frac{\pi^{k/2}}{\Gamma(\frac k2 +1)},
\;\;\;
\o_k = k \kappa_k = \frac{2\pi^{k/2}}{\Gamma(\frac k2)}. 
$$

\begin{theorem}\label{t2}
	Let $X_1,\dots,X_n$ be independent beta-distributed random points in $\BB^d$ with parameter $\b>-1$. Then
	\begin{align}
	\EE \Vol_d(P_{n,d}^{\beta})
	&=
	A_{n,d}^\b \int\limits_{-1}^{1}  \left(1-h^2\right)^{q} F_{1,\b+\frac{d-1}{2}}(h)^{n-d-1} \,\dint h,\label{eq:E_Vol_d_P}\\
	\EE \Vol_d(S_{n,d}^{\beta})
	&=
	2^{d+1} A_{n,d}^\b \int\limits_{0}^{1}  \left(1-h^2\right)^{q} \left(F_{1,\b+\frac{d-1}{2}}(h)-F_{1,\b+\frac{d-1}{2}}(-h)\right)^{n-d-1} \,\dint h,\nonumber \\
	\EE \Vol_d(Q_{n,d}^{\beta})
	&=
	D_{n,d}^\b + A_{n,d}^\b \int\limits_{0}^{1}  \left(1-h^2\right)^{q} F_{1,\b+\frac{d-1}{2}}(h)^{n-d-1}\, \dint h,\nonumber
	\end{align}
	where $q= (d+1)\left(\b-\frac 12\right) + \frac{d}{2}(d+3)$ and
	\begin{align}
	A_{n,d}^\b
	&=
	\frac {(d+1) \kappa_d}{2^d \pi^{\frac{d+1}{2}}} \binom {n}{d+1} \left(\beta + \frac {d+1}2\right)  \left(\frac{\Gamma\left(\frac{d+2}{2} + \b \right)}{\Gamma\left(\frac {d+3} 2 + \b\right)} \right)^{d+1},\label{eq:A_n_d_beta}\\
	D_{n,d}^\b
	&=\frac {\kappa_d}{2^n \pi^{\frac{d}{2}}} \binom {n}{d} \left(\frac{\Gamma\left(\frac{d+2}{2} + \b \right)}{\Gamma\left(\frac {d+3} 2 + \b\right)} \right)^{d}.\nonumber
	\end{align}
\end{theorem}


\begin{theorem}\label{t2_prime}
	Let $X_1,\dots,X_n$ be independent beta$^\prime$-distributed random points in $\RR^d$ with parameter $\b>\frac{d+1}{2}$. Then
	\begin{align*}
	\EE \Vol_d(\tilde P_{n,d}^{\beta})
	&=
	\tilde A_{n,d}^\b \int \limits_{-\infty}^{\infty} (1+h^2)^{-\tilde q} \tilde{F}_{1,\b-\frac{d-1}{2}}(h)^{n-d-1} \, \dint h,\\
	\EE \Vol_d(\tilde S_{n,d}^{\beta})
	&=
	2^{d+1} \tilde A_{n,d}^\b   \int\limits_{0}^{\infty}  (1+h^2)^{-\tilde q} \left(\tilde{F}_{1,\b-\frac{d-1}{2}}(h)-\tilde{F}_{1,\b-\frac{d-1}{2}}(-h)\right)^{n-d-1} \,\dint h,\\
	\EE \Vol_d(\tilde Q_{n,d}^{\beta})
	&=
	\tilde{D}_{n,d}^\b + \tilde A_{n,d}^\b  \int\limits_{0}^{\infty} (1+h^2)^{-\tilde q} \tilde{F}_{1,\b-\frac{d-1}{2}}(h)^{n-d-1} \,\dint h,
	\end{align*}
	where $\tilde{q}= (d+1)(\b + \frac 12) - \frac d2 (d+3)$ and
	\begin{align*}
	\tilde A_{n,d}^\b
	&=
	\frac {(d+1) \kappa_d}{2^d \pi^{\frac{d+1}{2}}} \binom {n}{d+1} \left(\beta - \frac {d+1}2\right)  \left(\frac{\Gamma\left(\b - \frac{d+1}{2} \right)}{\Gamma\left(\b - \frac {d} 2 \right)} \right)^{d+1},\\
	\tilde{D}_{n,d}^\b
	&=\frac {\kappa_d}{2^n \pi^{\frac{d}{2}}} \binom {n}{d} \left(\frac{\Gamma\left(\b - \frac{d+1}{2} \right)}{\Gamma\left(\b - \frac {d} 2 \right)} \right)^{d+1}.
	\end{align*}
\end{theorem}

The formulae for the expected intrinsic volumes can be obtained using the following proposition.
\begin{proposition}\label{prop:expected_intrinsic}
	The expected intrinsic volumes $\EE V_k(P_{n,d}^{\beta})$ and $\EE V_k(\tilde P_{n,d}^{\beta})$  for $k=1,\ldots,d$ are given by the formulae
	\begin{align*}
	\EE V_k(P_{n,d}^{\beta})
	&=
	\binom{d}{k} \frac{\k_d}{\k_k \k_{d-k}} \EE \Vol_k\left( P_{n,k}^{\beta + \frac{d-k}{2}}\right),\\
	\EE V_k(\tilde P_{n,d}^{\beta})
	&=
	\binom{d}{k} \frac{\k_d}{\k_k \k_{d-k}} \EE \Vol_k\left( \tilde P_{n,k}^{\beta - \frac{d-k}{2}}\right).
	\end{align*}
	These formulae hold if $P_{n,d}^{\beta}$, respectively $\tilde P_{n,d}^{\beta}$, is replaced by $S_{n,d}^{\beta}$ or $Q_{n,d}^{\beta}$, respectively $\tilde S_{n,d}^{\beta}$ or $\tilde Q_{n,d}^{\beta}$.
\end{proposition}

\subsection{Expected surface area and expected mean width}

In particular, \Cref{prop:expected_intrinsic} implies formulae for the expected surface area of the polytopes $P_{n,d}^{\beta}, S_{n,d}^{\beta}, Q_{n,d}^{\beta}, \tilde P_{n,d}^{\beta}, \tilde S_{n,d}^{\beta}$ and $\tilde Q_{n,d}^{\beta}$. For a convex set $K\subset\RR^d$ we use the abbreviation $S_{d-1}(K):=2V_{d-1}(K)$ to indicate the surface area of $K$.

\begin{corollary}\label{t3}
	Let $X_1,\dots,X_n$ be independent beta-distributed random points in $\BB^d$ with parameter $\b>-1$. Then
	\begin{align*}
	\EE S_{d-1}(P_{n,d}^{\beta})
	&=
	\g_d A_{n,d-1}^{\b+\frac{1}{2}} \int\limits_{-1}^{1}  \left(1-h^2\right)^{q} F_{1,\b+\frac{d-1}{2}}(h)^{n-d} \,\dint h,\\
	\nonumber\EE S_{d-1}(S_{n,d}^{\beta})
	&=
	2^{d} \g_d A_{n,d-1}^{\b+\frac{1}{2}} \int\limits_{0}^{1}  \left(1-h^2\right)^{q} \left(F_{1,\b+\frac{d-1}{2}}(h)-F_{1,\b+\frac{d-1}{2}}(-h)\right)^{n-d} \,\dint h,\\
	\nonumber\EE S_{d-1}(Q_{n,d}^{\beta})
	&=
	\g_d \left( D_{n,d-1}^{\b+\frac{1}{2}} + A_{n,d-1}^{\b+\frac{1}{2}} \int\limits_{0}^{1}  \left(1-h^2\right)^{q} F_{1,\b+\frac{d-1}{2}}(h)^{n-d} \,\dint h \right),
	\end{align*}
	where $q= d\b + \frac{d-1}{2}(d+2)$ and $\g_d=\frac{d\k_d}{\k_{d-1}}$. The constants $A_{n,d}^\b$ and $D_{n,d}^\b$ are the same as in \Cref{t2}.
\end{corollary}


\begin{corollary}\label{t3_prime}
	Let $X_1,\dots,X_n$ be independent beta$^\prime$-distributed random points in $\RR^d$ with parameter $\b>\frac{d+1}{2}$. Then
	\begin{align*}
	\EE S_{d-1}(\tilde P_{n,d}^{\beta})
	&=
	\g_d \tilde A_{n,d-1}^{\b-\frac{1}{2}} \int \limits_{-\infty}^{\infty} (1+h^2)^{-\tilde q} \tilde{F}_{1,\b-\frac{d-1}{2}}(h)^{n-d} \,\dint h,\\
	\EE S_{d-1}(\tilde S_{n,d}^{\beta})
	&=
	2^{d} \g_d \tilde A_{n,d-1}^{\b-\frac{1}{2}}   \int\limits_{0}^{\infty}  (1+h^2)^{-\tilde q} \left(\tilde{F}_{1,\b-\frac{d-1}{2}}(h)-\tilde{F}_{1,\b-\frac{d-1}{2}}(-h)\right)^{n-d} \,\dint h,\\
	\EE S_{d-1}(\tilde Q_{n,d}^{\beta})
	&=
	\g_d \left( \tilde{D}_{n,d-1}^{\b-\frac{1}{2}} + \tilde A_{n,d-1}^{\b-\frac{1}{2}}  \int\limits_{0}^{\infty} (1+h^2)^{-\tilde q} \tilde{F}_{1,\b-\frac{d-1}{2}}(h)^{n-d} \,\dint h \right),
	\end{align*}
	where $\tilde{q}= d\beta -\frac{d-1}{2} (d+2)$ and $\g_d=\frac{d\k_d}{\k_{d-1}}$. The constants $\tilde A_{n,d}^\b$ and $\tilde D_{n,d}^\b$ are the same as in \Cref{t2_prime}.
\end{corollary}

Similarly, we can find explicit formulae for the mean width of these random polytopes. We recall that the mean width $\mw_d(K)$ of a convex set $K\subset\RR^d$ is defined as the expected length of a projection of $K$ onto a uniformly chosen random line. The mean width is related to the first intrinsic volume by the formula $\mw_d(K)={2\kappa_{d-1}\over d\kappa_d}V_1(K)$.

\begin{corollary}\label{cor:width}
	Let $X_1,\dots,X_n$ be independent beta-distributed random points in $\BB^d$ with parameter $\b>-1$. Then
	\begin{align*}
	\EE\mw_{d}(P_{n,d}^{\beta})
	&=
	A_{n,1}^{\b+\frac{d-1}{2}} \int\limits_{-1}^{1}  \left(1-h^2\right)^{q} F_{1,\b+\frac{d-1}{2}}(h)^{n-2}\, \dint h,\\
	\EE \mw_{d}(S_{n,d}^{\beta})
	&=4A_{n,1}^{\b+\frac{d-1}{2}} \int\limits_{0}^{1}  \left(1-h^2\right)^{q} \left(F_{1,\b+\frac{d-1}{2}}(h)-F_{1,\b+\frac{d-1}{2}}(-h)\right)^{n-2} \,\dint h,\\
	\EE \mw_{d}(Q_{n,d}^{\beta})
	&=
	D_{n,1}^{\b+\frac{d-1}{2}} + A_{n,1}^{\b+\frac{d-1}{2}} \int\limits_{0}^{1}  \left(1-h^2\right)^{q} F_{1,\b+\frac{d-1}{2}}(h)^{n-2}\, \dint h,
	\end{align*}
	where $q= 2\beta+d$. The constants $A_{n,d}^\b$ and $D_{n,d}^\b$ are the same as in \Cref{t2}.
\end{corollary}

As in the case of the surface area, a different representation for $\EE\mw_d(P_{n,d}^{\beta})$ was previously given by Buchta, M\"uller and Tichy \cite{BuchtaMuellerTichy}.

\begin{corollary}
	Let $X_1,\dots,X_n$ be independent beta$^\prime$-distributed random points in $\RR^d$ with parameter $\b>\frac{d+1}{2}$. Then
	\begin{align*}
	\EE \mw_{d}(\tilde P_{n,d}^{\beta})
	&=
	\tilde A_{n,1}^{\b-\frac{d-1}{2}} \int \limits_{-\infty}^{\infty} (1+h^2)^{-\tilde q} \tilde{F}_{1,\b-\frac{d-1}{2}}(h)^{n-2} \,\dint h,\\
	\EE \mw_{d}(\tilde S_{n,d}^{\beta})
	&=
	4\tilde A_{n,1}^{\b-\frac{d-1}{2}}   \int\limits_{0}^{\infty}  (1+h^2)^{-\tilde q} \left(\tilde{F}_{1,\b-\frac{d-1}{2}}(h)-\tilde{F}_{1,\b-\frac{d-1}{2}}(-h)\right)^{n-2} \,\dint h,\\
	\EE \mw_{d}(\tilde Q_{n,d}^{\beta})
	&=
	\tilde{D}_{n,1}^{\b-\frac{d-1}{2}} + \tilde A_{n,1}^{\b-\frac{d-1}{2}}  \int\limits_{0}^{\infty} (1+h^2)^{-\tilde q} \tilde{F}_{1,\b-\frac{d-1}{2}}(h)^{n-2} \,\dint h,
	\end{align*}
	where $\tilde{q}= 2\beta-d$. The constants $\tilde A_{n,d}^\b$ and $\tilde D_{n,d}^\b$ are the same as in \Cref{t2_prime}.
\end{corollary}

\subsection{Random simplices and parallelotopes}

In the case when $n=d+1$, the random polytope $P_{d+1,d}^\beta$ is a simplex and a complete characterization of the distribution of its volume through its moments was obtained by Miles~\cite{Miles}; see also Kingman~\cite{Kingman} for the case $\beta=0$ in which the points are uniformly distributed in a ball.
We need to recall the results on the moments of the $d$-volume of a $d$-simplex and the $d$-volume of a $d$-parallelotope spanned by random points distributed according to either the beta-or beta$^\prime$-distribution. The following two results are due to Miles \cite[Equations (72) and (74)]{Miles} and Mathai \cite[Theorems 19.2.5 and 19.2.6]{Mathai}, respectively. However, while Miles \cite{Miles} only considers integer moments, we present an extension to moments of arbitrary (non-negative) real order, which is also needed elsewhere, for example in \cite{GroteKabluchkoThaele}.

In what follows we shall write $\Delta_d = \Delta_d\left( X_0,\dots,X_d \right)$ and $\nabla_d = \nabla_d\left( X_1,\dots,X_d \right)$ for the volume of the simplex with vertices $X_0,\ldots, X_d\in\RR^d$ and for the volume of the parallelotope spanned by $X_1,\ldots,X_{d}$, respectively.

\begin{proposition}{(Miles)}\label{lem:Miles}\ \\
	Let $X_0,\dots,X_d$ be i.i.d.\ random points in $\RR^d$.
	\begin{itemize}
		\item[(a)] If $X_0,\dots,X_d$ are distributed according to a beta-distribution on $\BB^d$ with parameter $\beta>-1$, then
		\[
		\EE_\b\left( \Delta_d^\kappa \right) = \left(d!\right)^{-\kappa} \frac{\G\left( \frac{d+1}{2}(2 \beta + d + \kappa) +1 \right)}{\G\left( \frac{d+1}{2}(2 \beta + d) + \frac{d \kappa}{2} +1 \right)} \left( \frac{\G\left( \frac{d}{2} + \b + 1 \right)}{\G\left( \frac{d+\kappa}{2} + \b + 1 \right)} \right)^{d+1} \prod_{i=1}^{d} \frac{\G\left( \frac{i+\kappa}{2} \right)}{\G\left( \frac{i}{2} \right)},
		\]
		for all $\kappa \in [0,\infty)$.
		\item[(b)] If $X_0,\dots,X_d$ are distributed according to a beta$^\prime$-distribution on $\RR^d$ with parameter $\beta>\frac d2$, then
		\[
		\widetilde{\EE}_{\b}\left( \Delta_d^\kappa \right) = \left(d!\right)^{-\kappa} \frac{\G\left( \left( \b - \frac{d}{2} \right)(d+1)-\frac{d}{2}\kappa \right)}{\G\left( \left( \b - \frac{d+\kappa}{2} \right)(d+1) \right)} \left( \frac{\G\left( \b -  \frac{d+\kappa}{2} \right)}{\G\left( \b - \frac{d}{2} \right)} \right)^{d+1} \prod_{i=1}^{d} \frac{\G\left( \frac{i+\kappa}{2} \right)}{\G\left( \frac{i}{2} \right)} \mbox{,}
		\]
		for all $\kappa \in \left[0, 2\b-d \right)$.
	\end{itemize}
\end{proposition}

\begin{remark}
	Note that in (b) we have indeed $-\frac{d}{2}\kappa$ instead of $+\frac{d}{2}\kappa$, as stated in \cite{Miles}. This typo has already been observed and corrected  by Chu \cite{Chu}, for example.
\end{remark}

The proof of \Cref{lem:Miles} will be based on the following moment formulae of Mathai \cite{Mathai_Random_Parallelotopes} for the volume of the random $d$-parallelotope spanned by $d$ points distributed in $\RR^d$ according to the beta- or the beta'-distribution. We emphasize that these formulae are known to be valid for arbitrary non-negative real moments; see~\cite{Mathai_Random_Parallelotopes}.

\begin{proposition}{(Mathai)}\label{lem:Mathai}\ \\
	Let $X_1,\dots,X_d$ be i.i.d.\ random points in $\RR^d$.
	\begin{itemize}
		\item[(a)] If $X_1,\dots,X_d$ are distributed according to a beta-distribution on $\BB^d$ with parameter $\beta>-1$, then
		\[
		\EE_\b\left( \nabla_d^\kappa \right) = \left( \frac{\G\left( \b + \frac{d}{2} + 1 \right)}{\G\left( \b + \frac{d+\kappa}{2} + 1 \right)} \right)^d \prod_{i=1}^{d} \frac{\G\left( \frac{d+\kappa-i+1}{2} \right) }{\G\left( \frac{d-i+1}{2} \right)},
		\]
		for all $\kappa \in [0,\infty)$.
		\item[(b)] If $X_1,\dots,X_d$ are distributed according to a beta$^\prime$-distribution on $\RR^d$ with parameter $\beta>\frac d2$, then
		\[
		\widetilde{\EE}_{\b}\left( \nabla_d^\kappa \right) = \left( \frac{\G\left( \b - \frac{d+\kappa}{2} \right)}{\G\left( \b - \frac{d}{2} \right)} \right)^d \prod_{i=1}^{d} \frac{\G\left( \frac{d+\kappa-i+1}{2} \right) }{\G\left( \frac{d-i+1}{2} \right)} \mbox{,}
		\]
		for all $\kappa \in \left[0, 2\b-d \right)$.
	\end{itemize}
\end{proposition}


\subsection{Expectation of the $T$-functional}
Fix $a,b\geq 0$. Recall that for a convex polytope $P\subset \RR^d$, we are interested in the functional
$$
T_{a,b}^{d,k}(P)= \sum_{F \in \mathcal{F}_k(P)} \eta^{a}(F) \Vol_k^b(F),
$$
where $\mathcal{F}_k(P)$ is the set of all $k$-dimensional faces of  $P$, and  $\eta(F)$ is the Euclidean distance from the affine hull ${\rm aff}(F)$ of the $k$-dimensional face $F$ to the origin.

\begin{theorem}\label{t1}
	Fix $a,b\geq 0$.
	Let $X_1,\dots,X_n$ be independent beta-distributed random points in $\BB^d$ with parameter $\b>-1$. Then
	\begin{align*}
	\EE T^{d,d-1}_{a,b} \left( P_{n,d}^\beta \right)
	&=
	C_{n,d}^{\b,b}\int\limits_{-1}^{1} \left|h\right|^a \left(1-h^2\right)^{d\b + \frac{d-1}{2}(d+b+1)} F_{1,\b+\frac{d-1}{2}}(h)^{n-d} \,\dint h,\\
	\EE T^{d,d-1}_{a,b} \left( S_{n,d}^\beta \right)
	&=
	2^d 
	C_{n,d}^{\b,b}
	\int\limits_{0}^{1} h^a \left(1-h^2\right)^{d\b + \frac{d-1}{2}(d+b+1)} \left(F_{1,\b+\frac{d-1}{2}}(h)-F_{1,\b+\frac{d-1}{2}}(-h)\right)^{n-d} \,\dint h,\\
	\EE T^{d,d-1}_{a,b} \left( Q_{n,d}^\beta \right)
	&=
	D_{n,d}^{\b,a,b} + C_{n,d}^{\b,b} \int\limits_{0}^{1} h^a \left(1-h^2\right)^{d\b + \frac{d-1}{2}(d+b+1)} F_{1,\b+\frac{d-1}{2}}(h)^{n-d} \,\dint h,
	\end{align*}
	where
	\begin{align*}
	C_{n,d}^{\b,b} &=\binom{n}{d} d! \k_d \EE_\b\left( \Delta_{d-1}^{b+1} \right) \left( \frac{c_{d,\b}}{c_{d-1,\b}} \right)^d,\\
	D_{n,d}^{\b,a,b}&=\ind{\{a=0\}} \binom{n}{d-1} \frac{d \k_d \EE_\b\left( \nabla_{d-1}^{b+1} \right)}{2^{n-d+1}((d-1)!)^b} \left(\frac{c_{d,\b}}{c_{d-1,\b}}\right)^{d-1} \mbox{.}
	\end{align*}
\end{theorem}

\begin{remark}
	Since the polytopes $S_{n,d}^\b$ and $Q_{n,d}^\b$ always contain the origin, we can decompose them into simplices of the form $[0,x_1,\ldots,x_d]$, where $[x_1,\ldots,x_d]$ runs through all facets of the corresponding polytope. It follows that the expected volume of these polytopes is given by
	$$
	\EE \Vol_d(S_{n,d}^\b) = {\frac 1{d}} \EE T_{1,1}^{d,d-1}(S_{n,d}^\b),
	\quad
	\EE \Vol_d(Q_{n,d}^\b) = {\frac 1{d}} \EE T_{1,1}^{d,d-1}(Q_{n,d}^\b).
	$$
	Then, \Cref{t1}  yields the formulae
	\begin{align*}
	\EE \Vol_d(S_{n,d}^\b)
	&=\frac{2^d}{d}
	C_{n,d}^{\b,1}
	\int\limits_{0}^{1} h \left(1-h^2\right)^{d\b + \frac{d-1}{2}(d+2)} \left(F_{1,\b+\frac{d-1}{2}}(h)-F_{1,\b+\frac{d-1}{2}}(-h)\right)^{n-d} \,\dint h,\\
	\EE \Vol_d(Q_{n,d}^\b)
	&=\frac 1 {d} D_{n,d}^{\b,1,1} + \frac 1 {d} C_{n,d}^{\b,1} \int\limits_{0}^{1} h \left(1-h^2\right)^{d\b + \frac{d-1}{2}(d+2)} F_{1,\b+\frac{d-1}{2}}(h)^{n-d} \,\dint h.
	\end{align*}
	This method does not work for $P_{n,d}^\b$ because this polytope need not contain $0$. However, one can consider the following analogue of the $T$-functional defined for $d$-dimensional polytopes $P\subset \RR^d$:
	$$
	T_{1,1,\pm}^{d,d-1}(P)= \sum_{F \in \mathcal{F}_{d-1}(P)} \eta_\pm(F) \Vol_{d-1}(F),
	$$
	where $\eta_\pm(F)$ is the distance from the affine hull of the face $F$ to the origin taken positive if the origin and the polytope $P$ are on the same side of the face $F$, and negative otherwise. Then, it is easy to see that
	$$
	\EE \Vol_d(P_{n,d}^\b) = \frac 1{d} \EE T_{1,1,\pm}^{d,d-1}(P_{n,d}^\b).
	$$
	The expected value of $T_{1,1,\pm}^{d,d-1}(P_{n,d}^\b)$ can be computed in the same way as in \Cref{t1} yielding  the following formula:
	$$
	\EE \Vol_d(P_{n,d}^\b)
	=\frac{1}{d}
	C_{n,d}^{\b,1}\int\limits_{-1}^{1}  h \left(1-h^2\right)^{d\b + \frac{d-1}{2}(d+2)} F_{1,\b+\frac{d-1}{2}}(h)^{n-d}\, \dint h.
	$$
	The equivalence of these formulae to those given in \Cref{t2} can be shown using partial integration. Anyway, we shall give an alternative proof of \Cref{t2} in Section~\ref{sec:proof_exp_volumes}.
\end{remark}

The analogue of Theorem~\ref{t1} in the case of beta$^\prime$ densities can be stated as follows.
\begin{theorem}\label{t1_prime}
	Fix $a,b\geq 0$. Let $X_1,\dots,X_n$ be independent beta$^\prime$-distributed random points in $\RR^d$ with parameter $\b$ that satisfies $2d \beta > (d-1)(d+b+1) + a +1$. Then
	\begin{align*}
	\EE T^{d,d-1}_{a,b} \left( \tilde{P}_{n,d}^\beta \right)
	&=
	\tilde{C}^{\beta, b}_{n,d} \int \limits_{-\infty}^{\infty} \left|h\right|^a
	\left(1+h^2\right)^{-d\b+\frac{d-1}{2}\left( d+b+1 \right)}
	\tilde{F}_{1,\b - \frac{d-1}{2}}(h)^{n-d}\, \dint h,\\
	\EE T^{d,d-1}_{a,b} \left( \tilde{S}_{n,d}^\beta \right)
	&=
	2^d \tilde{C}_{n,d}^{\b,b} \int\limits_{0}^{\infty} h^a \left(1+h^2\right)^{-d\b+\frac{d-1}{2}\left( d+b+1 \right)} \left(\tilde{F}_{1,\b-\frac{d-1}{2}}(h)-\tilde{F}_{1,\b-\frac{d-1}{2}}(-h)\right)^{n-d} \,\dint h,\\
	\EE T^{d,d-1}_{a,b} \left( \tilde{Q}_{n,d}^\beta \right)
	&=
	\tilde{D}_{n,d}^{\b,a,b} + \tilde{C}_{n,d}^{\b,b} \int\limits_{0}^{\infty} h^a \left(1+h^2\right)^{-d\b+\frac{d-1}{2}\left( d+b+1 \right)} \tilde{F}_{1,\b-\frac{d-1}{2}}(h)^{n-d} \,\dint h,
	\end{align*}
	where
	\begin{align*}
	\tilde{C}_{n,d}^{\b,b}&=\binom{n}{d} d! \k_d \widetilde{\EE}_{\b}\left( \Delta_{d-1}^{b+1} \right) \left( \frac{\tilde{c}_{d,\b}}{\tilde{c}_{d-1,\b}} \right)^d,\\
	\tilde{D}_{n,d}^{\b,a,b}&=\ind_{\{a=0\}} \binom{n}{d-1} \frac{d \k_d \widetilde{\EE}_{\b}\left( \nabla_{d-1}^{b+1} \right)}{2^{n-d+1}((d-1)!)^b} \left(\frac{\tilde{c}_{d,\b}}{\tilde{c}_{d-1,\b}}\right)^{d-1}.
	\end{align*}
\end{theorem}

\begin{remark}
Theorem \ref{t1} immediately implies exact formulae for the expected facet numbers of random beta-polytopes. Namely, setting $a=b=0$, it follows from the definition of $T_{a,b}^{d,d-1}$ that we have
\begin{equation}\label{eq:FacetNumber}
\EE \mathbf{f}_{d-1}(P_{n,d}^\b) = \EE T_{0,0}^{d,d-1}(P_{n,d}^\b).
\end{equation}
Here, $\mathbf{f}_k(P) = |\mathcal{F}_k(P)|$ denotes the number of $k$-dimensional faces of the polytope $P$. The expected facet numbers of $S_{n,d}^\b$, $Q_{n,d}^\b$, $\tilde P_{n,d}^\b$, $\tilde S_{n,d}^\b$ and $\tilde Q_{n,d}^\b$ follow analogously.
\end{remark}

\section{Special cases}\label{sec:SpecialCases}

\subsection{Uniform distribution in the ball}
The beta distribution with $\beta=0$ is just the uniform distribution in the ball. As a special case of \Cref{t2} we immediately obtain the following
\begin{corollary}\label{prop:uniform_ball}
	For all $d\in \NN$ we have
	\begin{align*}
	\EE \Vol_d(P_{n,d}^0)
	&=
	\frac{(d+1)^2\k_d}{2^{d+1} \pi^{\frac{d+1}{2}}} \binom {n}{d+1}\left(\frac{\Gamma(\frac {d+2}2)}{\Gamma(\frac{d+3}{2})}\right)^{d+1} \int\limits_{-1}^{1} (1-h^2)^{\frac{d^2+2d-1} 2} F_{1,\frac{d-1}{2}}(h)^{n-d-1} \,\dint h,\\
	\EE \Vol_d(S_{n,d}^0)
	&=
	\frac{(d+1)^2\k_d}{\pi^{\frac{d+1}{2}}} \binom {n}{d+1}\left(\frac{\Gamma(\frac {d+2}2)}{\Gamma(\frac{d+3}{2})}\right)^{d+1}\\
	&\qquad\qquad\qquad\qquad\times \int\limits_{0}^{1} (1-h^2)^{\frac{d^2+2d-1}2} \left(F_{1,\frac{d-1}{2}}(h)-F_{1,\frac{d-1}{2}}(-h)\right)^{n-d-1} \, \dint h.
	\end{align*}
\end{corollary}

Now, we specialize these formulae to dimensions $d=2$ and $d=3$. We start with the $2$-dimensional case and remark that the formula for $\EE \Vol_2 (P_{n,2}^0)$ has previously been obtained by Efron, see Equation (7.13) in \cite{Efron65}.

\begin{corollary}
	In dimension $d=2$, we have
	\begin{align*}
	\EE \Vol_2 (P_{n,2}^0) &= \frac {16}{3(2\pi)^{n-1}} \binom n3 \int\limits_0^{2\pi} (h-\sin h)^{n-3} \sin^8 \left(\frac h2 \right)\,\dint h,\\
	\EE \Vol_2 (S_{n,2}^0) &=\frac {32}{3\cdot \pi^{n-1}} \binom n3 \int\limits_0^{\pi} (h+\sin h)^{n-3} \cos^8 \left(\frac h2 \right)\,\dint h.
	\end{align*}
\end{corollary}
\begin{proof}
Since $F_{1,{1\over 2}}(h)={1\over 2}+{1\over\pi}(h\sqrt{1-h^2}+\arcsin h)$, $|h|\leq 1$, the first result follows from the substitution $h=-\cos{y\over 2}$ and by relabeling of $y$ by $h$. The second result follows from the observation that $F_{1,{1\over 2}}(h)-F_{1,{1\over 2}}(-h)={2\over \pi}(h\sqrt{1-h^2}+\arcsin h)$, $|h|\leq 1$, by the substitution $h=\sin{y\over 2}$ and by relabeling of $y$ by $h$.
\end{proof}

As anticipated above, the computation of $\EE \Vol_2 (P_{n,2}^0)$ is due to Efron \cite{Efron65}. Buchta~\cite{BuchtaEllipsoid} obtained the following more elegant representation:
$$
\EE \Vol_2 (P_{n,2}^0) = \pi + \frac {1}{3(2\pi)^{n-1}} \int\limits_0^{2\pi} (h-\sin h)^n \sin h \,\dint h.
$$
We continue with the $3$-dimensional case and again remark that the formula for $\EE \Vol_3 (P_{n,3}^0)$ corresponds to Equation (7.14) in \cite{Efron65}. We also refer to Buchta~\cite{BuchtaEllipsoid}, who derived another representation for the integral.

\begin{corollary}
	In dimension $d=3$, we have
	\begin{align*}
\EE \Vol_3 (P_{n,3}^0) &=  \frac{27 \pi}{4^{n+1}} \binom n4 \int\limits_{-1}^1 (1-h)^7(1+h)^{2n-1}(2-h)^{n-4}\, \dint h,\\
	\EE \Vol_3 (S_{n,3}^0) &= \frac{27 \pi}{2^{n+2}} \binom n4 \int\limits_{0}^1 (1-h^2)^{7} h^{n-4} (3-h^2)^{n-4}\, \dint h.
	\end{align*}
\end{corollary}
\begin{proof}
To obtain $\EE \Vol_3 (P_{n,3}^0)$ we choose $d=3$ in \Cref{prop:uniform_ball} to obtain $F_{1,1}(h)={3\over 4}({2\over 3}+h-{h^3\over 3})={1\over 4}(h-2)(1+h)^2$, $|h|\leq 1$ and hence, by factorization of the function under the integral,
\begin{align*}
\EE \Vol_3 (P_{n,3}^0) &= {27\pi\over 1024}{1\over 4^{n-4}}{n\choose 4}\int_{-1}^1(1-h^2)^7(1+h)^{2(n-4)}(h-2)^{n-4}\,\dint h\\
&={27\pi\over 4^{n+1}}{n\choose 4}\int_{-1}^1(1-h)^7(1+h)^{2n-1}(2-h)^{n-4}\,\dint h.
\end{align*}
Similarly, $F_{1,1}(h)-F_{1,1}(-h)={1\over 2}h(3-h^2)$, $|h|\leq 1$, and the result for $\EE \Vol_3 (S_{n,3}^0)$ follows.
\end{proof}


Using the classical Efron identity, we can obtain the following formulae for the expected number of vertices.
\begin{proposition}
	The expected number of vertices of $P_{n,d}^0$ and $S_{n,d}^0$ is given by
	\begin{align*}
	\EE \mathbf{f}_{0} (P_{n,d}^0) = n \left(1 - \frac {\EE \Vol_d (P_{n-1,d}^0)} {\kappa_d} \right),
	\quad
	\EE \mathbf{f}_{0} (S_{n,d}^0) = 2n \left(1 - \frac {\EE \Vol_d (S_{n-1,d}^0)} {\kappa_d} \right),
	\end{align*}
	where the expected volumes on the right-hand side were given in \Cref{prop:uniform_ball}.
\end{proposition}
\begin{proof} The first formula is the classical Efron identity, see \cite{Efron65}. The second formula can be obtained as follows:
	\begin{align*}
	\EE \mathbf{f}_{0} (S_{n,d}^0)
	&=
	2n \PP(X_1 \text{ is a vertex of } S_{n,d}^0)\\
	&=
	2n \PP(X_1 \notin [\pm X_2,\ldots,\pm X_n])
	=
	2n \left(1 - \frac {\EE \Vol_d (S_{n-1,d}^0)} {\kappa_d} \right).
	\end{align*}
	This completes the argument.
\end{proof}

Furthermore, we can state results for the expected surface area and the expected mean width of $P_{n,2}^0$ and $S_{n,2}^0$ as well as of $P_{n,3}^0$ and $S_{n,3}^0$ that follow from Corollaries~\ref{t3} and~\ref{cor:width}. 

\begin{corollary}
	In dimension $d=2$, we have
	\begin{align*}
	\EE S_1 (P_{n,2}^0) &= \frac{1}{3\pi 2^{n-4}}{n\choose 2} \int\limits_{-\pi/2}^{\pi/2} \cos^5h \left(1 + \frac{\sin(2h) + 2h }{\pi} \right)^{n-2} \dint h,\\
	\EE S_1 (S_{n,2}^0) &= \frac{64}{3\pi} {n\choose 2} \int\limits_{0}^{\pi/2} \cos^5h \left(\frac {\sin(2h)+2h}{\pi} \right)^{n-2} \dint h.
	\end{align*}
\end{corollary}

\begin{corollary}
	In dimension $d=3$, we have
	\begin{align*}
	\EE S_2 (P_{n,3}^0) &= \frac{81\pi}{4^{n+1}} \binom n 3 \int\limits_{-1}^{1} (1-h)^5(1+h)^{2n-1}(2-h)^{n-3}\,\dint h,\\
	\EE S_2 (S_{n,3}^0) &= \frac{81\pi }{2^{n+2}} \binom n3 \int\limits_{0}^{1} (1-h^2)^5 h^{n-3}(3-h^2)^{n-3} \,\dint h,\\
	\EE\mw_3(P_{n,3}^0) &= \frac{9}{4^{n}} \binom n 2 \int\limits_{-1}^{1}(1-h)^3(1+h)^{2n-1}(2-h)^{n-2}\,\dint h,\\
	\EE\mw_3(S_{n,3}^0) &= \frac{9}{2^{n}} \binom n 2\int\limits_{0}^{1}(1-h^2)^3h^{n-2}(3-h^2)^{n-2}\,\dint h.
	\end{align*}
\end{corollary}
\begin{proof}
Observing that $\gamma_3A_{n,2}^{1\over 2}={81\pi\over 256}{n\choose 3}$ and that $F_{1,1}(h)={1\over 4}(2+3h-h^3)={1\over 4}(1+h)^2(2-h)$, $|h|\leq 1$, the formulae for $\EE S_2 (P_{n,3}^0)$ and $\EE S_2 (S_{n,3}^0)$ follow from Corollary \ref{t3}. The values for $\EE\mw_3(P_{n,3}^0)$ and $\EE\mw_3(S_{n,3}^0)$ can be deduced from Corollary \ref{cor:width} and the fact that $A_{n,1}^{1}={9\over 16}{n\choose 2}$.
\end{proof}

Finally, let us discuss the explicit formulae for the expected number of facets of $P_{n,2}^0,P_{n,3}^0,S_{n,2}^0$ and $S_{n,3}^0$. They follow by specializing \Cref{t1} to the case implied by \eqref{eq:FacetNumber}.

\begin{corollary}
In dimension $d=2$, we have
\begin{align*}
\EE\mathbf{f}_1(P_{n,2}^0)
&=
\frac {1}{3\pi 2^{n-4}}{n\choose 2} \int\limits_{-\pi/2}^{\pi/2}\cos^4 h\left(1 + \frac{\sin(2h) + 2h }{\pi} \right)^{n-2}\,\dint h,\\
\EE\mathbf{f}_1(S_{n,2}^0)
&=
\frac{64}{3\pi} {n\choose 2} \int\limits_{0}^{\pi/2} \cos^4h \left(\frac {\sin(2h)+2h}{\pi} \right)^{n-2} \dint h
\end{align*}
\end{corollary}


\begin{corollary}
In dimension $d=3$ it holds that
	\begin{align*}
	\EE\mathbf{f}_2(P_{n,3}^0) &= {630\over 4^{n+1}}{n\choose 3}\,\int\limits_{-1}^1(1-h)^4(1+h)^{2n-2}(2-h)^{n-3}\,\dint h,\\
	\EE\mathbf{f}_2(S_{n,3}^0) &= {315\over 2^{n+1}}{n\choose 3}\int\limits_{0}^1(1-h^2)^{4}h^{n-3}(3-h^2)^{n-3}\,\dint h.
	\end{align*}
\end{corollary}

\subsection{Uniform distribution on the sphere}
We shall argue below that the uniform distribution on the sphere $\SS^{d-1}$ is the weak limit of the beta distribution, as $\b \downarrow -1$. Let $X_1,\ldots,X_n$ be independent random points chosen uniformly on the sphere $\SS^{d-1}$. Consider the random polytopes
$$
P_{n,d} := [X_1,\ldots,X_n]
\quad\text{and}\quad
S_{n,d} := [\pm X_1,\ldots,\pm X_n].
$$
We claim that the expected volumes of these polytopes can be obtained by formally taking $\beta= -1$ in Theorem~\ref{t2}:

\begin{corollary}\label{prop:uniform_sphere}
	For all $d\geq 2$ we have
	\begin{align*}
	\EE \Vol_d(P_{n,d})
	&=
	\frac{(d^2-1)\k_d}{2^{d+1} \pi^{\frac{d+1}{2}}}\binom {n}{d+1}\left(\frac{\Gamma(\frac d2)}{\Gamma(\frac{d+1}{2})}\right)^{d+1} \int\limits_{-1}^{1} (1-h^2)^{\frac{d^2-3}2} F_{1,\frac{d-3}{2}}(h)^{n-d-1}\, \dint h,\\
	\EE \Vol_d(S_{n,d})
	&=
	\frac{(d^2-1)\k_d}{\pi^{\frac{d+1}{2}}}\binom {n}{d+1}\left(\frac{\Gamma(\frac d2)}{\Gamma(\frac{d+1}{2})}\right)^{d+1}\\
	&\qquad\qquad\qquad\qquad\qquad\times \int\limits_{0}^{1} (1-h^2)^{\frac{d^2-3}2} \left(F_{1,\frac{d-3}{2}}(h)-F_{1,\frac{d-3}{2}}(-h)\right)^{n-d-1} \,\dint h.
	\end{align*}
\end{corollary}
An equivalent formula for $\EE \Vol_d(P_{n,d})$ was obtained in~\cite[p.~228]{BuchtaMuellerTichy} by a different method; see also~\cite{Mueller} for an asymptotic result. A different representation for $\EE \Vol_d(P_{n,d})$ was obtained by Affentranger~\cite{Affentranger88}. Again, let us specialize the result to the $2$- and $3$-dimensional case.

\begin{corollary}
	In dimension $d=2$, we have
	\begin{align*}
	\EE \Vol_2 (P_{n,2}) &= \frac 3 {\pi^{2}} \binom n3 \int\limits_0^\pi \sin^2 h \left(\frac{h}\pi\right)^{n-3} \dint h,\\
	\EE \Vol_2 (S_{n,2}) &= \frac{24}{\pi^2}\binom n3 \int\limits_0^{\pi/2} \sin^2 h \left(1 - \frac{2h}\pi\right)^{n-3} \dint h.
	\end{align*}
\end{corollary}
\begin{proof}
	It follows from~\eqref{eq:def_F_d_beta} that $F_{1,-1/2}(h) = \frac 2 \pi \arcsin \sqrt{\frac {h+1}{2}}$ for $|h|\leq 1$. The claim follows by the change of variables $h = -\cos y$ (for the first integral) or $h= \cos y$ (for the second integral), by elementary transformations and by renaming $y$ by $h$.
\end{proof}
\begin{corollary}
	In dimension $d=3$, we have
	$$
	\EE \Vol_3 (P_{n,3}) = \frac {4\pi}{3} \frac{(n-1)(n-2)(n-3)}{(n+1)(n+2)(n+3)},
	\quad
	\EE \Vol_3 (S_{n,3}) = \frac {4\pi}{3} \frac{n(n-2)}{(n+1)(n+3)}.
	$$
\end{corollary}
The first of these formulae is due to Affentranger~\cite{Affentranger88}.
\begin{proof}
	We used that $F_{1,0}(h) = \frac 12 (1+h)$ for $|h|\leq 1$ together with the formulae
	\begin{align*}
	&\int\limits_{-1}^{1} (1 - h^2)^3 2^{-A}(1 + h)^A \dint h = \frac {3\cdot 2^8}{(4 + A) (5 + A) (6 + A) (7 + A)},\\
	&\int\limits_{0}^{1} (1 - h^2)^3 h^A \dint h = \frac {48}{(1 + A) (3 + A) (5 + A) (7 + A)},
	\end{align*}
	where $A>0$, and straightforward transformations.
\end{proof}

We are also able to give explicit formulae for the expected surface areas and the mean widths in dimensions $d=2$ and $d=3$.

\begin{corollary}
In dimension $d=2$, we have
\begin{align*}
\EE S_1 (P_{n,2})
&=
{4\over \pi}{n\choose 2} \int\limits_{-\pi/2}^{\pi/2} \left( \frac{1}{2} + \frac{h}{\pi} \right)^{n-2} \cos h\,\dint h,\\
\EE S_1 (S_{n,2})
&={2^{n+2}\over\pi^{n-1}}{n\choose 2} \int_0^{\pi/2}h^{n-2} \cos h\, \dint h.
\end{align*}
\end{corollary}

\begin{proof}
In view of \Cref{t3}, the formula for $\EE S_1 (P_{n,2})$ follows from the fact that
$q=0$, $\gamma_2A_{n,1}^{-{1\over 2}}={2\over\pi}n(n-1)$, $F_{1,-1/2}(h)={1\over 2}+{1\over\pi}\arcsin h$, $|h|\leq 1$,
by applying the substitution $h=\sin y$ and by renaming $y$ by $h$. The case of $\EE S_1 (S_{n,2})$ is similar.
\end{proof}

\begin{corollary}
	In dimension $d=3$, we have
	\begin{alignat*}{4}
	&&\EE S_2 (P_{n,3}) &= 4\pi \frac{(n-1)(n-2)}{(n+1)(n+2)},  \qquad  &\EE S_2(S_{n,3}) & = {4\pi}\frac{n-1}{n+2},\\
	&&\EE \mw_3 (P_{n,3}) &= 2\,\frac{n-1}{n+1},  \qquad  &\EE \mw_3(S_{n,3}) &= 2\,\frac{n}{n+1} .
	\end{alignat*}
\end{corollary}
\begin{proof}
	 To obtain $\EE S_2 (P_{n,3})$ and $\EE S_2(S_{n,3})$ we note that $\gamma_3A_{n,2}^{-{1\over 2}}={\pi\over 16}n(n-1)(n-2)$ and $F_{1,0}(h)={h+1\over 2}$, $|h|\leq 1$. Since $q=2$ in this case and since
	\begin{align*}
	\int_{-1}^1 (1-h^2)^2\,\Big({h+1\over 2}\Big)^{n-3}\,\dint h &= {64\over n(n+1)(n+2)},\\
	\int_{0}^1 (1-h^2)^2\,h^{n-3}\,\dint h &= {8\over n(n-2)(n+2)},
	\end{align*}
	the formulae follow again from \Cref{t3}. The computations for $\EE \mw_3 (P_{n,3})$ and $\EE \mw_3(S_{n,3})$ are similar.
\end{proof}

Note that particular values of $\EE S_1 (P_{n,2})$, $\EE S_2 (P_{n,3})$ and $\EE \mw_3 (P_{n,3})$, for $n=2,3,4$ and $n=3,4,5$, respectively, were calculated by Buchta, M\"uller and Tichy \cite{BuchtaMuellerTichy}. We have recovered these special values and found general simple closed expressions for the $3$-dimensional case that are valid for all $n\geq 4$ and that were not available before.

As above, we finally present formulae for the expected number of facets of $P_{n,2},P_{n,3},S_{n,2}$ and $S_{n,3}$. Before, we notice that, with probability $1$,
$$
\mathbf{f}_{0}(P_{n,d}) = n\qquad\text{and}\qquad\mathbf{f}_{0}(S_{n,d}) = 2n,
$$
since each of the $n$ points $X_1,\ldots,X_n$ is almost surely a vertex of $P_{n,d}$ and each of the $2n$ points $\pm X_1,\ldots,\pm X_n$ is almost surely a vertex of the symmetric random polytope $S_{n,d}$.

\begin{corollary}
	In dimension $d=2$, we have
	\begin{align*}
	\mathbf{f}_1(P_{n,2}) &= n\qquad\text{almost surely},\\
	\mathbf{f}_1(S_{n,2}) &= 2n\qquad\!\!\!\text{almost surely}.
	\end{align*}
\end{corollary}	
\begin{proof}
Indeed, $\mathbf{f}_1(P_{n,2})=\mathbf{f}_{0}(P_{n,2})=n$ and $\mathbf{f}_1(S_{n,2})=\mathbf{f}_{0}(S_{n,2}) = 2n$, as argued above.
\end{proof}

\begin{corollary}
	In dimension $d=3$ it holds that
	\begin{align*}
	\mathbf{f}_2(P_{n,3}) &= 2(n-2)\qquad\text{almost surely},\\
	\mathbf{f}_2(S_{n,3}) &= 4(n-1)\qquad\text{almost surely}.
	\end{align*}
\end{corollary}
\begin{proof}
	We observe that with probability one the random beta-polytope $P_{n,3}$ is a simplicial polytope, that is, all faces of $P_{n,3}$ are almost surely triangles. This implies that $2{\bf f}_1(P_{n,3})=3{\bf f}_2(P_{n,3})$, which together with Euler's polyhedron formula ${\bf f}_0(P_{n,3})-{\bf f}_1(P_{n,3})+{\bf f}_2(P_{n,3})=2$ leads to
	$$
	{\bf f}_2(P_{n,3}) = 2({\bf f}_0(P_{n,3})-2) = 2(n-2).
	$$
	Similarly, we have that
	$$
	{\bf f}_2(S_{n,3}) = 2({\bf f}_0(S_{n,3})-2) = 2(2n-2) = 4(n-1).
	$$
	with probability one.
\end{proof}

\subsection{Expected facet number of spherical convex hulls}
Denote by
$$
\SSd_+=\{x=(x_1,\ldots,x_{d+1})\in\RRd1:x_{d+1}\geq 0\}
$$
the $d$-dimensional closed upper half-sphere in $\RRd1$. For $\alpha>-1$ we consider the distribution on $\SSd_+$ whose density with respect to the uniform distribution on $\SSd_+$ is given by
$$
\hat f_{d,\alpha}(x) = c_{d,\alpha}\,x_{d+1}^\alpha, \qquad x=(x_1,\ldots,x_{d+1})\in\SSd_+.
$$
Here, $c_{d,\alpha}$ is a normalizing constant.  The spherical convex hull $\hat P_{n,d}^\alpha$ of $n$ independent random points on $\SSd_+$ distributed according to the density $\hat f_{d,\alpha}$ has been studied in \cite{BonnetEtAl} (for general $\alpha$) and \cite{BaranyEtAlHalfsphere} (for the special case $\alpha=0$). In particular, it has been shown in~\cite[Section~5]{BonnetEtAl} that the expected number of facets of the spherical random polytope $\hat P_{n,d}^\alpha$ coincides with that of $\tilde P_{n,d}^\beta$ for the choice $\beta={1\over 2}(\alpha+d+1)$. Thus, \eqref{eq:FacetNumber} also yields a new and explicit formula for the expected number of facets of spherical convex hull $\hat P_{n,d}^\alpha$.

\begin{proposition}
Let $\alpha>-1$ and put $\beta={1\over 2}(\alpha+d+1)$. Then
\begin{align*}
\EE\mathbf{f}_{d-1}(\hat P_{n,d}^\alpha) = \EE\mathbf{f}_{d-1}(\tilde P_{n,d}^\beta)
=
\tilde{C}_{n,d}^{\frac{\alpha+d+1}{2},0}\int\limits_0^\pi \sin^{\alpha d + d-1} h \,\tilde{F}_{1,\frac \alpha 2+1}(\cot h)^{n-d}\, \dint h.
\end{align*}
	In particular, if $\alpha=0$, then
	\begin{align*}
	\EE\mathbf{f}_{d-1}(\hat P_{n,d}^0) = \EE\mathbf{f}_{d-1}(\tilde P_{n,d}^{d+1\over 2})  = \binom{n}{d} \frac{2 \omega_d}{\omega_{d+1}} \int\limits_0^\pi \left( 1-\frac{h}{\pi} \right)^{n-d} \sin^{d-1} h\, \dint h.
	\end{align*}
\end{proposition}

For $\alpha=0$ this formula was obtained in \cite{BaranyEtAlHalfsphere}, see Theorem 3.1 there.

\begin{proof}
	The first formula follows from \Cref{t1_prime} with $b=0$ there together with the substitution $h=\cot y$ and then by renaming $y$ by $h$. The second formula follows from the first one by observing that $\tilde F_{1,1}(\cot h) = 1-{h\over\pi}$, $h\in[0,\pi]$, and that $\tilde C_{n,d}^{{d+1\over 2},0}={n\choose d}{2\o_d\over\o_{d+1}}$.
\end{proof}

\subsection{Gaussian distribution}
	By letting $\beta\to +\infty$ in either beta or beta' model, it is possible to recover the Gaussian distribution on $\RR^d$ as limiting case. Indeed, if the random point $X$ has $d$-dimensional beta distribution, then $X \sqrt {2\beta}$ has the density
	$$
	(2\b)^{-d/2} f_{d,\b} \left(\frac X{\sqrt {2\b}}\right) = \frac 1 {(2\b \pi)^{\frac d2}}  \frac{ \G\left( \frac{d}{2} + \b + 1 \right) }{ \G\left( \b+1 \right) } \left( 1- \frac{\| x\|^2}{2\b}\right)^\b \ind_{\{\|x\| < \sqrt {2\b}\}} \tobeta  \frac {\eee^{- \frac 12 \|x\|^2}}{(2\pi)^{\frac d2}},
	$$
	and similarly for the beta' model. The convex hull of an i.i.d.\ Gaussian sample is called the Gaussian polytope. The expected volume of the Gaussian polytope was computed by Efron~\cite{Efron65}; see also~\cite{GroteThaele,HugMunsoniusReitzner} and~\cite{KabluchkoZaporozhets} and for further recent results about Gaussian polytopes.

\section{Auxiliary results}\label{sec:AuxResults}

\subsection{Tools from integral geometry}
The Euclidean norm of a vector  $x\in \RR^d$ is denoted by $\|x\|$ or by $\|x\|_d$  if we would like to emphasize the role of dimension $d$.
Let $\BB^d_r(x)$ be the ball centered at the point $x\in\RR^d$ and having  radius $r>0$.  In particular, we write $\BB^d=\BB^d_1(0)$ for the Euclidean unit ball. Its boundary, the unit $(d-1)$-dimensional sphere, is denoted by $\SS^{d-1}$. Furthermore, let $\lambda_d$ denote the Lebesgue measure on $\RR^d$, while $\sigma_{d-1}$ be the surface Lebesgue measure on $\SS^{d-1}$.

Let $G(d,k)$ be the set of all $k$-dimensional linear subspaces of $\RR^d$. The unique Haar probability measure on $G(d,k)$ invariant with respect to the natural action of the orthogonal group is denoted by $\nu_k$. Analogously, $A(d,k)$ stands for the set of all $k$-dimensional affine subspaces of $\RR^d$ with the measure $\mu_k$ defined by
\begin{equation}\label{eq:def_Haar_on_aff_Grass}
\mu_k(\cdot) = \int\limits_{G(d,k)} \int\limits_{L^\perp} \ind_{\{ L+x \in \,\cdot\, \}} \,\lambda_{L^\perp}(\dint x) \nu_k(\dint L) \mbox{.}
\end{equation}
Similarly as in the case of $\RR^d$, we mean by $\left\|\cdot\right\|_L$, $\BB^k_L$, $\SS^{k-1}_L$, $\lambda_L$ and $\sigma_L$ the norm, the unit ball, the unit sphere, the Lebesgue measure and the Lebesgue surface measure on the unit sphere, respectively, in a linear subspace $L \in G(d,k)$, and by $\left\|\cdot\right\|_E$, $\BB^k_E$, $\SS^{k-1}_E$, $\lambda_E$ and $\sigma_E$ the respective objects in an affine subspace $E \in A(d,k)$. This is the same notation as used in the book~\cite{SW}, to which we refer for further information.

 Furthermore, we use the symbol $\Delta_k(x_0,\dots,x_k)=\Vol_k([x_0,\dots,x_k])$ for the $k$-dimensional vol\-ume of the $k$-simplex spanned by the points $x_0,\dots,x_k \in \RR^d$ and $\nabla_k(x_1,\dots,x_k)$ for the $k$-dimensional volume of the parallelotope spanned by the vectors $x_1,\dots,x_k \in \RR^d$, respectively. These two quantities are related by
\begin{equation}\label{e7}
\Delta_d(x_0,\dots,x_k)=\frac{1}{k!} \nabla_d(x_1-x_0,\dots,x_k-x_0) \mbox{.}
\end{equation}

A particularly valuable tool in the main proof will be the affine Blaschke-Petkantschin formula; see \cite[Theorem 7.2.7]{SW}.
\begin{lemma}{(Affine Blaschke-Petkantschin formula)}\label{thm:ABP}\ \\
	Let $q\in \{ 1,\dots,d \}$ and $\varphi:\left( \RR^d \right)^{q+1} \rightarrow \RR$ be non-negative and measurable. Then
	\begin{align*}
	&\int\limits_{(\RR^d)^{q+1}}\varphi(x_1,\ldots,x_{q+1})\,\l_d^{q+1}(\dint(x_1,\ldots,x_{q+1})) \\
	&\qquad= b_{d,q} (q!)^{d-q} \int\limits_{A(d,q)}\int\limits_{E^{q+1}}\varphi(x_1,\ldots,x_{q+1})\,\Delta_q(x_1,\ldots,x_{q+1})^{d-q}\,\l_E^{q+1}(\dint(x_1,\ldots,x_{q+1}))\,\mu_q(\dint E)
	\end{align*}
	with the constant $b_{d,q}$ given by
	\[
	b_{d,q} := {\o_{d-q+1}\cdots\o_d\over\o_1\cdots\o_q},
	\]
and where $\o_k= k \k_k=\sigma_{k-1}(\SS^{k-1})$ is the surface area of $\SS^{k-1}$, $k\in\NN$.
\end{lemma}

Furthermore, we will make use of Kubota's formula; see \cite[Equations (6.11) and~(5.5)]{SW}.
\begin{lemma}{(Kubota's formula)}\label{thm:Kub}\ \\
	Let $K\subset \RR^d$ be a nonempty, compact, convex set. Let $k\in \{ 1,\dots,d-1 \}$. Then,
	\[
	V_k\left( K \right) = \binom{d}{k} \frac{\k_d}{\k_k \k_{d-k}} \int\limits_{G(d,k)} \Vol_k(K|L) \,\nu_k(\dint L) \mbox{,}
	\]
	where $K|L$ stands for the orthogonal projection of $K$ onto $L$.
\end{lemma}

\subsection{Projections of beta- and beta'-densities}
In the next lemma, we show that the beta-and beta$^\prime$-distributions on $\RR^d$ yield distributions of the same type (but with different parameters) when projected onto arbitrary linear subspaces. Fix some $k\in \{1,\dots,d-1\}$ and let $L \in G(d,k)$ be a $k$-dimensional linear subspace.
Since the beta- and beta$^\prime$-distribution are rotationally invariant it suffices to investigate the $k$-dimensional subspace
$$
L=\{x \in \RR^d: x_{k+1}=\dots=x_d = 0\}
$$
for the results to come. Furthermore, we will identify $L$ and $\RR^k$.   We will often use these observations implicitly and refrain from restating them at every step of the proof.

\begin{lemma}\label{lem1}
 Denote by $\pi_L: \RR^d\rightarrow L$ the orthogonal projection onto $L$.
	\begin{itemize}
		\item[(a)] If the random variable $X$ has density $f_{d,\b}$, then $\pi_L (X)$ has density $f_{k,\b+\frac{d-k}{2}}$.
		\item[(b)] If the random variable $X$ has density $\tilde{f}_{d,\b}$, then $\pi_L(X)$ has density $\tilde{f}_{k,\b-\frac{d-k}{2}}$.
	\end{itemize}
\end{lemma}

\begin{proof}
Both, for (a) and (b) it suffices to consider the case $k=d-1$ because then we can argue by induction.
So, let $L=\{ x \in \RR^d: x_d=0 \}$, which is identified with $\RR^{d-1}$, as explained before.

Let us prove (a). Fix some $x^*= (x_1^*,\ldots, x_{d-1}^*) \in \BB^{d-1}$ with Euclidean norm $r := \|x^*\|_{d-1}\in [0,1)$. The pre-images of $x^*$ under the projection map $\pi_L$ have the form $x= (x_1^*,\ldots,x_{d-1}^*, x_d)$ with $x_d\in\RR$, but since we are interested only in $x\in \BB^d$, we obtain the restriction $|x_d| \leq \sqrt{1-r^2}$.  It holds that $\|x\|_{d}^2 = r^2 + x_d^2$ .
Thus, the density of $\pi_L(X)$ at $x^*$ is given by
\begin{align*}
c_{d,\b} \int \limits_{-\sqrt{1-r^2}}^{+\sqrt{1-r^2}}  \left( 1- \|x\|_d^2 \right)^\b \dint x_d
&=
c_{d,\b} \int \limits_{-\sqrt{1-r^2}}^{+\sqrt{1-r^2}}  \left( 1- r^2 - x_d^2 \right)^\b \dint x_d\\
	&=
c_{d,\b} \left( 1-r^2 \right)^\b \int \limits_{-\sqrt{1-r^2}}^{+\sqrt{1-r^2}} \left( 1- \frac{x_d^2}{1-r^2} \right)^\b \dint x_d \\
	&=c_{d,\b}  \left( 1-r^2 \right)^{\b+\frac 12}\int \limits_{-1}^{1} \left( 1- y^2 \right)^\b \dint y,
\end{align*}
where we used the transformation $y=x_d/\sqrt{1-r^2}$. Hence, $\pi_L(X)$ has density $f_{d-1,\b+\frac{1}{2}}$ and we do not even need to check that
$
c_{d,\b} \int\limits_{-1}^{1} \left( 1- y^2 \right)^\b\,\dint y = c_{d-1, \b + \frac 12},
$
because the result must be a probability density.
Inductive application of this result yields the desired statement for arbitrary dimensions $k$.

In the case of the beta$^\prime$-distribution we can apply almost  the same argument. Fix some $x^* = (x_1^*,\ldots, x_{d-1}^*) \in \RR^{d-1}$ with Euclidean norm $r := \|x^*\|_{d-1} \geq 0$. The pre-images of $x^*$ under the projection $\pi_L$ have the form $(x_1^*,\ldots,x_{d-1}^*, x_d)$ with $x_d\in\RR$.
Then, the density of $\pi_L(X)$ at $x^*$ is given by
\begin{align*}
\tilde c_{d,\b} \int \limits_{-\infty}^{+\infty}  \left( 1 + \|x\|_d^2 \right)^{-\b} \dint x_d
&=
\tilde{c}_{d,\b} \int\limits_{-\infty}^{\infty} \left( 1 + r^2 + x_d^2 \right)^{-\b} \dint x_d \\
&=
\tilde{c}_{d,\b} \left( 1 + r^2 \right)^{-\b} \int\limits_{-\infty}^{\infty} \left( 1 + \frac{x_d^2}{1+r^2} \right)^{-\b} \dint x_d \\
&=
\tilde c_{d,\b}  \left( 1+r^2 \right)^{-(\b-\frac 12)}\int \limits_{-\infty}^{\infty} \left( 1+ y^2 \right)^{-\b} \dint y,
\end{align*}
where we used the transformation $y=x_d/\sqrt{1+r^2}$.	It follows that $\pi_L(X)$ has density $\tilde{f}_{d-1,\b-\frac{1}{2}}$. The statement in the case of general dimension $k$ again follows by induction.
\end{proof}

\subsection{Probability contents of half-spaces and slabs}
The probability content $\mathrm{PC}_\b(A)$, respectively $\widetilde{\mathrm{PC}}_{\b}(A)$,  of a measurable set $A\subset \RR^d$ is  the probability that a random vector with a beta, respectively beta'-distribution, attains a value in $A$.
Consider the hyperplane $E_0=\{ x \in \RR^d: x_d=0 \}\in G(d,d-1)$. Denote by $E_h=\{ x \in \RR^d: x_d=h \} \in A(d,d-1)$, $h \in \RR$, the affine hyperplanes parallel to $E_0$. Later on, we will need the $d$-dimensional probability content of the half-spaces
$$
E_h^+ = \{x\in\RR^d : x_d \geq h\} \qquad\text{and}\qquad    E_h^- = \{x\in\RR^d : x_d \leq h\}
$$
with respect to the beta- and beta'-distribution. 
Hence, we are interested in the quantities
\[
	\mathrm{PC}_\b\left(E_h^+\right) = \int\limits_{h}^{1} \int\limits_{E_t} f_{d,\beta}(x) \,\lambda_{E_t}(\dint x) \dint t ,\qquad \mathrm{PC}_\b\left(E_h^-\right) = \int\limits_{-1}^{h} \int\limits_{E_t} f_{d,\beta}(x) \,\lambda_{E_t}(\dint x) \dint t
\]
and
\[
	\widetilde{\mathrm{PC}}_{\b}\left(E_h^+\right) = \int\limits_{h}^{\infty} \int\limits_{E_t} \tilde{f}_{d,\beta}(x) \,\lambda_{E_t}(\dint x) \dint t , \qquad \widetilde{\mathrm{PC}}_{\b}\left(E_h^-\right) = \int\limits_{-\infty}^{h} \int\limits_{E_t} \tilde{f}_{d,\beta}(x) \,\lambda_{E_t}(\dint x) \dint t
\]
respectively. Furthermore, for $h\geq 0$ we will be interested in the probability content of the slab between $E_h$ and $E_{-h}$, that is,
\[
	\mathrm{PC}_{\b}\left(E_h^- \cap E_{-h}^+\right) = \int\limits_{-h}^{h} \int\limits_{E_t} f_{d,\beta}(x) \,\lambda_{E_t}(\dint x) \dint t
\]
and
\[
	\widetilde{\mathrm{PC}}_{\b^\prime}\left(E_h^- \cap E_{-h}^+\right) = \int\limits_{-h}^{h} \int\limits_{E_t} \tilde{f}_{d,\beta}(x) \,\lambda_{E_t}(\dint x) \dint t,
\]
respectively. Recall from~\eqref{eq:def_F_d_beta} and~\eqref{eq:def_F_d_beta_prime}  that we denote the distribution function of the one-dimensional beta- and beta$^\prime$-distribution by $F_{1,\b}$ and $\tilde{F}_{1,\b}$, respectively.

\begin{lemma}\label{lem2}
Consider the affine hyperplane $E_h=\{ x \in \RR^d: x_d=h \}$ with $h \in \RR$. In the case of the beta-distribution on $\BB^d$ with parameter $\beta>-1$, we have
	\[
	 \mathrm{PC}_\b\left(E_h^+\right)  = 1- F_{1,\b+\frac{d-1}{2}}(h),\qquad \mathrm{PC}_\b\left(E_h^-\right)  = F_{1,\b+\frac{d-1}{2}}(h), \qquad h\in [-1,1],
	\]
	and
	\[
		\mathrm{PC}_\b\left(E_h^- \cap E_{-h}^+\right)  = F_{1,\b+\frac{d-1}{2}}(h) - F_{1,\b+\frac{d-1}{2}}(-h), \qquad h\in [0,1].
	\]
Similarly, in the case of the beta$^\prime$-distribution on $\RR^d$ with parameter $\beta>\frac d2$, it holds that
	\[
	\mathrm{PC}_{\b^\prime}\left(E_h^+\right) = 1- \tilde{F}_{1,\b-\frac{d-1}{2}}(h),\qquad \mathrm{PC}_{\b^\prime}\left(E_h^-\right) = \tilde{F}_{1,\b-\frac{d-1}{2}}(h), \qquad h\in\RR,
	\]
	and
	\[
	\mathrm{PC}_{\b^\prime}\left(E_h^- \cap E_{-h}^+\right) = \tilde{F}_{1,\b-\frac{d-1}{2}}(h) - \tilde{F}_{1,\b-\frac{d-1}{2}}(-h), \qquad h\geq 0.
	\]
\end{lemma}

\begin{proof}
Let $X$ be a random point with density $f_{d,\b}$. In order to calculate the probability contents of $E_h^+$ and $E_h^-$ we project $X$ onto the line $$
L:=E_0^\perp = \{x_1=\ldots=x_{d-1} =0\},
$$
that is, the orthogonal complement of $E_0$. Clearly, $X\in E_h^+$ is equivalent to $\pi_L(X) \geq h$, whereas $X\in E_h^-$ is equivalent to $\pi_L(X) \leq h$.  From \Cref{lem1} we know that $\pi_{L}(X)$ has the one-dimensional density $f_{1,\b+\frac{d-1}{2}}$. Hence,
	\[
		\mathrm{PC}_\beta(E_h^-) = \int\limits_{-\infty}^{h} f_{1,\b+\frac{d-1}{2}}(x)\, \dint x  = c_{1,\b+\frac{d-1}{2}} \int\limits_{-1}^{h} \left( 1- x^2 \right)^{\b+\frac{d-1}{2}} \dint x = F_{1,\b+\frac{d-1}{2}}(h).
	\]
To complete the proof in the beta case, observe that $\mathrm{PC}_\b(E_h^+)=1-\mathrm{PC}_\beta(E^-)$ and $\mathrm{PC}_\b(E_h^- \cap E_{-h}^+)=\mathrm{PC}_\beta(E_h^-)-\mathrm{PC}_\beta(E_{-h}^-)$.

Similarly, if $X$ is a random variable with density $\tilde{f}_{d,\b}$, then, by \Cref{lem1}, $\pi_{L}(X)$ has density $\tilde{f}_{1,\b-\frac{d-1}{2}}$ and we get the corresponding results for the beta$^\prime$-distribution.
\end{proof}

\subsection{Random simplices in affine hyperplanes}

We will need the moments of the volume of a random simplex chosen according to the density $f_{d,\beta}$ or $\tilde f_{d,\beta}$ restricted to an affine hyperplane. Recall the definitions of $\EE_\beta(\Delta_d^\kappa)$ and $\widetilde{\EE}_\beta(\Delta_d^\kappa)$ from \Cref{lem:Miles}.

\begin{lemma}\label{lem3}
Let $E \in A(d,d-1)$ be an affine hyperplane at distance $h$ from the origin. In the case of the beta-distribution with parameter $\beta>-1$, for all $h\in [0,1]$ and $\kappa\in[0,\infty)$ we have
\begin{align*}
		\int \limits_{E^d} \Delta_{d-1}^{\kappa} \left( x_1,\dots,x_d \right) &\left(\prod_{i=1}^{d} f_{d,\b}(x_i)\right) \lambda_E^d\left( \dint \left( x_1,\dots,x_d \right) \right) \\ &=  \frac{c_{d,\b}^d}{c_{d-1,\b}^d} \left(1-h^2\right)^{d\b + \frac{d-1}{2}(d+\kappa)} \EE_\b\left(\Delta_{d-1}^\kappa\right).
\end{align*}
Similarly, in the case of the beta$^\prime$-distribution with parameter $\beta>\frac d2$, for all $h\geq 0$ and $\kappa \in [0, 2\beta -d)$ we have
\begin{align*}
		\int \limits_{E^d} \Delta_{d-1}^{\kappa}\left( x_1,\dots,x_d \right) &\left(\prod_{i=1}^{d} \tilde{f}_{d,\b}(x_i) \right)\lambda_E^d\left( \dint \left( x_1,\dots,x_d \right) \right) \\ &= \frac{\tilde{c}_{d,\b}^d}{\tilde{c}_{d-1,\b}^d} \left(1+h^2\right)^{-d\b + \frac{d-1}{2}(d+\kappa)} \widetilde{\EE}_{\b}\left(\Delta_{d-1}^\kappa\right).
\end{align*}
\end{lemma}

\begin{proof}
Without loss of generality we take $E=\{z \in \RR^d: z_d=h\} \in A(d,d-1)$. Consider also the linear hyperplane $L=\{z \in \RR^d: z_d=0\} \in G(d,d-1)$ which is parallel to $E$. With $z^*=\pi_L(z)$, for $z \in E$, we have that $\left\| z \right\|_d^2 = \left\| z^* \right\|_{d-1}^2 + h^2$. Hence, for all $h\in [0,1]$,
	\begin{align*}
		&\int \limits_{E^d} \Delta_{d-1}^{\kappa}\left( x_1,\dots,x_d \right) \prod_{i=1}^{d} f_{d,\b}(x_i) \lambda_E^d\left( \dint \left( x_1,\dots,x_d \right) \right)\\
		&= c_{d,\b}^d \int \limits_{\left( E \cap \BB^d \right)^d} \Delta_{d-1}^{\kappa}\left( x_1,\dots,x_d \right)  \prod_{i=1}^{d} \left(1-\left\| x_i \right\|_d^2 \right)^\b \lambda_E^d\left( \dint \left( x_1,\dots,x_d \right) \right) \\
		&= c_{d,\b}^d \int \limits_{\left( L \cap \BB^d(\sqrt{1-h^2}) \right)^d} \Delta_{d-1}^{\kappa}\left( {x}_1^*,\dots,{x}_d^* \right)  \prod_{i=1}^{d} \left(1-\left\| {x}_i^* \right\|_{d-1}^2 - h^2 \right)^\b \lambda_L^d\left( \dint \left( {x}_1^*,\dots,{x}_d^* \right) \right) \\
		&= c_{d,\b}^d \left(1-h^2\right)^{d\b} \int \limits_{\left( L \cap \BB^d(\sqrt{1-h^2}) \right)^d} \Delta_{d-1}^{\kappa}\left( {x}_1^*,\dots,{x}_d^* \right)  \prod_{i=1}^{d} \left(1-\frac{\left\| {x}_i^* \right\|_{d-1}^2}{1-h^2}\right)^\b \lambda_L^d\left( \dint \left( {x}_1^*,\dots,{x}_d^* \right) \right) \\
		&= c_{d,\b}^d \left(1-h^2\right)^{d\b+\frac{d-1}{2}(d+\kappa)} \int \limits_{\left( \BB^{d-1} \right)^d} \Delta_{d-1}^{\kappa}\left( y_1,\dots,y_d \right)  \prod_{i=1}^{d} \left(1-\left\| y_i \right\|_{d-1}^2\right)^\b \lambda_L^d\left( \dint \left( y_1,\dots,y_d \right) \right) \\
		&=\frac{c_{d,\b}^d}{c_{d-1,\b}^d} \left(1-h^2\right)^{d\b+\frac{d-1}{2}(d+\kappa)} \int \limits_{\left( \BB^{d-1} \right)^d} \Delta_{d-1}^{\kappa}\left( y_1,\dots,y_d \right)  \prod_{i=1}^{d} f_{d-1,\b}(y_i) \lambda_L^d\left( \dint \left( y_1,\dots,y_d \right) \right) \\
		&=\frac{c_{d,\b}^d}{c_{d-1,\b}^d} \left(1-h^2\right)^{d\b+\frac{d-1}{2}(d+\kappa)} \EE_\b\left(\Delta_{d-1}^\kappa\right),
	\end{align*}
where we used the transformation $y_i=\left(1-h^2\right)^{-\frac{1}{2}}{x}^*_i$ for $i=1,\dots,d$.
	
The result for the beta$^\prime$-distribution is derived analogously by using the transformation $y_i=\left(1+h^2\right)^{-\frac{1}{2}}{x}_i^*$, for $i=1,\dots,d$, and by suitably adapting the range of integration.
\end{proof}

\section{Proof of Theorems~\ref{t1} and~\ref{t1_prime}}\label{sec:Proof1}
Let $X_1,\dots,X_n$ be independent and beta-distributed random points in $\BB^d$ with parameter $\beta$. We start with the polytope $P_{n,d}^\b = [X_1,\ldots,X_n]$. We have
	\begin{align*}
	\EE &T^{d,d-1}_{a,b} \left( P_{n,d}^\beta \right) \\
	&= \EE \left( \sum_{ 1\leq i_1<\dots<i_d \leq n } \ind \left( [X_{i_1},\dots,X_{i_d}] \in \mathcal{F}_{d-1}\left( P_{n,d}^\beta \right) \right) \eta^a\left([X_{i_1},\dots,X_{i_d}]\right) \Delta^b_{d-1}\left( X_{i_1},\dots,X_{i_d} \right) \right) \\
	&= \binom{n}{d} \EE \left( \ind\left( [X_1,\dots,X_d] \in \mathcal{F}_{d-1}\left( P_{n,d}^\beta \right) \right) \eta^a\left([X_1,\dots,X_d]\right) \Delta^b_{d-1}\left( X_1,\dots,X_d \right) \right) \\
	&= \binom{n}{d} \int \limits_{\left( \RR^d \right)^d} \PP\left( [x_1,\dots,x_d] \in \mathcal{F}_{d-1}\left( P_{n,d}^\beta \right) \Big| X_1=x_1,\dots,X_d=x_d \right) \eta^a\left([x_1,\dots,x_d]\right) \\
	&\qquad\qquad\qquad\times \Delta^b_{d-1}\left( x_1,\dots,x_d \right) \left( \prod_{i=1}^{d} f_{d,\beta}(x_i) \right) \lambda_d^d\left( \dint \left( x_1,\dots,x_d \right) \right),
	\end{align*}
where in the last step we conditioned on the event $\{X_1=x_1,\ldots, X_d=x_d\}$ and used the formula for the total probability.
Applying the affine Blaschke-Petkantschin formula stated in \Cref{thm:ABP} with $q=d-1$, we obtain
\begin{align*}
\EE &T^{d,d-1}_{a,b}
=
\binom{n}{d} (d-1)! \frac{d \kappa_{d}}{2} \int \limits_{A(d,d-1)} \int \limits_{E^d} \PP\left( [x_1,\dots,x_d] \in \mathcal{F}_{d-1}\left( P_{n,d}^\beta \right) \Big| X_1=x_1,\dots,X_d=x_d \right)\\
&\qquad \times \eta^a\left([x_1,\dots,x_d]\right) \Delta_{d-1}^{b+1}\left( x_1,\dots,x_d \right) \left( \prod_{i=1}^{d} f_{d,\beta}(x_i) \right) \lambda_E^d\left( \dint \left( x_1,\dots,x_d \right) \right) \mu_{d-1}(\dint E) \mbox{.}
\end{align*}

We denote by $h$ the distance from $E= \mathrm{aff} (x_1,\ldots,x_d)$ to the origin. Note that the conditional probability in the right-hand side is the probability that all $X_{d+1},\dots,X_{n}$ lie in either the half-space $E^+$ or the half-space $E^-$. By using rotation invariance of the density $f_{d,\beta}$ we may assume that $E$ has the form $E_h$ and then apply \Cref{lem2} to get
\begin{align*}
		&\PP\left( [x_1,\dots,x_d] \in \mathcal{F}_{d-1}\left( P_{n,d}^\beta \right) \Big| X_1=x_1,\dots,X_d=x_d \right)\\
		&\qquad\qquad\qquad\qquad = \left(1-F_{1,\b+\frac{d-1}{2}}(h)\right)^{n-d}+ F_{1,\b+\frac{d-1}{2}}(h)^{n-d}.
\end{align*}
	Since the integrand is rotationally invariant we can use formula~\eqref{eq:def_Haar_on_aff_Grass} to rewrite the integration over $A(d,d-1)$ as
	\begin{align*}
	\EE T^{d,d-1}_{a,b} \left( P_{n,d}^\beta \right) &=\binom{n}{d} d! \kappa_d \int \limits_{0}^{1} \left( \left(1-F_{1,\b+\frac{d-1}{2}}(h)\right)^{n-d}+ F_{1,\b+\frac{d-1}{2}}(h)^{n-d} \right) h^a \\
	&\qquad\qquad\times \int \limits_{E^d} \Delta_{d-1}^{b+1}\left( x_1,\dots,x_d \right) \left( \prod_{i=1}^{d} f_{d,\beta}(x_i) \right) \lambda_E^d\left( \dint \left( x_1,\dots,x_d \right) \right) \dint h \\
	&=\binom{n}{d} d! \k_d \left( \frac{c_{d,\b}}{c_{d-1,\b}} \right)^d \EE_\b\left( \Delta_{d-1}^{b+1} \right) \int\limits_{0}^{1} h^a \left(1-h^2\right)^{d\b + \frac{d-1}{2}(d+b+1)} \\
	& \qquad\qquad\times \left( \left(1-F_{1,\b+\frac{d-1}{2}}(h)\right)^{n-d}+ F_{1,\b+\frac{d-1}{2}}(h)^{n-d} \right) \dint h,
	\end{align*}
	where the second equality follows from \Cref{lem3}. In the next step we exploit the identities $f_{1,\b+\frac{d-1}{2}}(h) = f_{1,\b+\frac{d-1}{2}}(-h)$ and $1 - F_{1,\b+\frac{d-1}{2}}(h) = F_{1,\b+\frac{d-1}{2}}(-h)$ to rewrite the integral as
	\begin{align*}
		\EE T^{d,d-1}_{a,b} \left( P_{n,d}^\beta \right) &= C_{n,d}^{\b,b} \int\limits_{-1}^{1} \left|h\right|^a \left(1-h^2\right)^{d\b + \frac{d-1}{2}(d+b+1)} F_{1,\b+\frac{d-1}{2}}(h)^{n-d} \,\dint h,
	\end{align*}
	where
\begin{equation}\label{eq:def_C_n_d}
C_{n,d}^{\b,b}=\binom{n}{d} d! \k_d \EE_\b\left( \Delta_{d-1}^{b+1} \right) \left( \frac{c_{d,\b}}{c_{d-1,\b}} \right)^d.
\end{equation}
	
Slight adaptation of this proof yields the result for the symmetric polytope $S_{n,d}^\b= [\pm X_1,\ldots,\pm X_d]$ as follows:
	\begin{align*}
		\EE &T^{d,d-1}_{a,b} \left( S_{n,d}^\beta \right) \\
		&= \EE \Bigg( \sum_{ 1\leq i_1<\dots<i_d \leq n } \sum_{j_1,\dots,j_d \in \{0,1\}} \ind\left( \left[(-1)^{j_1}X_{i_1},\dots,(-1)^{j_d}X_{i_d}\right] \in \mathcal{F}_{d-1}\left( S_{n,d}^\beta \right) \right) \\
		&\qquad\qquad\times \eta^a\left(\left[(-1)^{j_1}X_{i_1},\dots,(-1)^{j_d}X_{i_d}\right]\right) \Delta^b_{d-1}\left( (-1)^{j_1}X_{i_1},\dots,(-1)^{j_d}X_{i_d} \right) \Bigg)  \\
		&= 2 ^d \EE \left( \sum_{ 1\leq i_1<\dots<i_d \leq n } \ind\left( [X_{i_1},\dots,X_{i_d}] \in \mathcal{F}_{d-1}\left( S_{n,d}^\beta \right) \right) \eta^a\left([X_{i_1},\dots,X_{i_d}]\right) \Delta^b_{d-1}\left( X_{i_1},\dots,X_{i_d} \right) \right) \\
		&= 2^d \binom{n}{d} \EE \left( \ind\left( [X_1,\dots,X_d] \in \mathcal{F}_{d-1}\left( S_{n,d}^\beta \right) \right) \eta^a\left([X_1,\dots,X_d]\right) \Delta^b_{d-1}\left( X_1,\dots,X_d \right) \right) \\
		&= 2^d \binom{n}{d} \int \limits_{\left( \RR^d \right)^d} \PP\left( [X_1,\dots,X_d] \in \mathcal{F}_{d-1}\left( S_{n,d}^\beta \right) \Big| X_1=x_1,\dots,X_d=x_d \right) \eta^a\left([x_1,\dots,x_d]\right) \\
		&\qquad\qquad\qquad\times \Delta^b_{d-1}\left( x_1,\dots,x_d \right) \left( \prod_{i=1}^{d} f_{d,\beta}(x_i) \right) \lambda_d^d\left( \dint \left( x_1,\dots,x_d \right) \right).
\end{align*}
Applying the affine Blaschke-Petkantschin formula, see \Cref{thm:ABP}, we get
\begin{align*}
\EE &T^{d,d-1}_{a,b} \left( S_{n,d}^\beta \right) =
2^d \binom{n}{d} (d-1)! \frac{d \kappa_{d}}{2} \\
&\qquad\times\int \limits_{A(d,d-1)} \int \limits_{E^d} \PP\left( [x_1,\dots,x_d] \in \mathcal{F}_{d-1}\left( S_{n,d}^\beta \right) \Big| X_1=x_1,\dots,X_d=x_d \right)\\
		&\qquad\times \eta^a\left([x_1,\dots,x_d]\right) \Delta_{d-1}^{b+1}\left( x_1,\dots,x_d \right) \left( \prod_{i=1}^{d} f_{d,\beta}(x_i) \right) \lambda_E^d\left( \dint \left( x_1,\dots,x_d \right) \right) \mu_{d-1}(\dint E) \mbox{.}
	\end{align*}
Let $h\in [0,1]$ be the distance from $E= \mathrm{aff} (x_1,\ldots,x_d)$ to the origin. The conditional probability on the right-hand side is the probability that the points $X_{d+1},\dots,X_{n}$ lie between the hyperplanes $E^+$ and $-E^-$.  Therefore, by \Cref{lem2},
\begin{align*}
	&\PP\left( [X_1,\dots,X_d] \in \mathcal{F}_{d-1}\left( S_{n,d}^\beta \right) \Big| X_1=x_1,\dots,X_d=x_d \right) \\ &\qquad\qquad\qquad\qquad\qquad= \left(F_{1,\b+\frac{d-1}{2}}(h)-F_{1,\b+\frac{d-1}{2}}(-h)\right)^{n-d}.
\end{align*}
	All the remaining steps are exactly the same as before (except for the last step, where we cannot exploit symmetry this time). Hence,
\begin{align*}
		\EE T^{d,d-1}_{a,b} \left( S_{n,d}^\beta \right) &= 2^d C_{n,d}^{\b,b} \int\limits_{0}^{1} h^a  \left(1-h^2\right)^{d\b + \frac{d-1}{2}(d+b+1)}\\
		&\qquad\qquad\times \left(F_{1,\b+\frac{d-1}{2}}(h)-F_{1,\b+\frac{d-1}{2}}(-h)\right)^{n-d} \dint h,
\end{align*}
where
$
C_{n,d}^{\b,b}
$
is the same as in~\eqref{eq:def_C_n_d}.

The derivation of the result for the polytope $Q_{n,d}^\b= [0, X_1,\ldots,X_n]$ needs a case distinction. Namely, we need to distinguish between facets that contain $0$ as a vertex and facets which do not. Furthermore, facets containing $0$ only contribute to $\EE T_{a,b}^{d,d-1}\left( Q_{n,d}^\b \right)$ in the case that the parameter $a$ equals zero. Hence,
	\begin{align*}
		\EE& T_{a,b}^{d,d-1}\left( Q_{n,d}^\b \right) \\
		&= \EE \Bigg( \ind_{\{a = 0\}} \Bigg( \sum_{ 1\leq i_1<\dots<i_{d-1} \leq n } \ind\left( [0,X_{i_1},\dots,X_{i_{d-1}}] \in \mathcal{F}_{d-1}\left( Q_{n,d}^\beta \right) \right) \Delta^b_{d-1}\left( 0,X_{i_1},\dots,X_{i_{d-1}} \right) \Bigg) \\
		&\qquad + \sum_{1\leq i_1<\dots<i_d \leq n} \ind\left( [X_{i_1},\dots,X_{i_d}] \in \mathcal{F}_{d-1}\left( Q_{n,d}^\beta \right) \right) \eta^a\left([X_{i_1},\dots,X_{i_d}]\right) \Delta^b_{d-1}\left( X_{i_1},\dots,X_{i_d} \right) \Bigg) \\
		&=\ind_{\{a=0\}} \binom{n}{d-1} \EE \left( \ind\left( [0,X_1,\dots,X_{d-1}] \in \mathcal{F}_{d-1}\left( Q_{n,d}^\beta \right) \right) \Delta^b_{d-1}\left( 0,X_1,\dots,X_{d-1} \right) \right) \\
		&\qquad + \binom{n}{d} \EE\left( \ind\left( [X_1,\dots,X_d] \in \mathcal{F}_{d-1}\left( Q_{n,d}^\beta \right) \right) \eta^a\left([X_1,\dots,X_d]\right) \Delta^b_{d-1}\left( X_1,\dots,X_d \right)  \right) \\
		&=\ind_{\{a=0\}} \binom{n}{d-1}  \int\limits_{\left(\RR^d\right)^{d-1}} \PP\left( [0,x_1,\dots,x_{d-1}] \in \mathcal{F}_{d-1}\left(Q_{n,d}^\b\right) \Big| X_1=x_1,\dots,X_{d-1}=x_{d-1} \right) \\
		&\qquad\qquad\times \Delta_{d-1}^{b}(0,x_1,\dots,x_{d-1}) \left( \prod_{i=1}^{d-1} f_{d,\b}(x_i) \right) \lambda_d^{d-1}\left(\dint (x_1,\dots,x_{d-1})\right) \\
		&\qquad + \binom{n}{d} \int\limits_{\left(\RR^d\right)^d} \PP\left( [x_1,\dots,x_d] \in \mathcal{F}_{d-1}\left(Q_{n,d}^\b\right) \Big| X_1=x_1,\dots,X_d=x_d \right) \eta^a\left([x_1,\dots,x_d]\right) \\
		&\qquad\qquad\times \Delta_{d-1}^{b}(x_1,\dots,x_d) \left( \prod_{i=1}^{d} f_{d,\b}(x_i) \right)  \lambda_d^d\left(\dint (x_1,\dots,x_d)\right).
	\end{align*}
Now observe that if $X_1=x_1,\ldots, X_{d-1}=x_{d-1}$, then $[0,x_1,\dots,x_{d-1}]$ is a face of $Q_{n,d}^\b$ if and only if the points $X_{d}, \ldots,X_n$ are on the same side of the hyperplane passing through $0,x_1,\dots,x_{d-1}$.  It immediately follows that
	\[
		\PP\left( [0,x_1,\dots,x_{d-1}] \in \mathcal{F}_{d-1}\left(Q_{n,d}^\b\right) \Big| X_1=x_1,\dots,X_{d-1}=x_{d-1} \right) = 2 \cdot 2^{-(n-d+1)} = 2^{-(n-d)}.
	\]	
	Furthermore, by denoting $E=\mathrm{aff}(X_1,\dots,X_d) \in A(d,d-1)$, we immediately see that the probability $\PP\left( [X_1,\dots,X_d] \in \mathcal{F}_{d-1}\left(Q_{n,d}^\b\right) \Big| X_1=x_1,\dots,X_d=x_d \right)$ is the same as the probability that all points $\pi_{E^\perp}(X_1),\dots,\pi_{E^\perp}(X_d)$ lie on the same side of $\pi_{E^\perp}(E)$ on which $0$ lies. By \Cref{lem2}, we therefore have
	\[
		\PP\left( [X_1,\dots,X_d] \in \mathcal{F}_{d-1}\left(Q_{n,d}^\b\right) \Big| X_1=x_1,\dots,X_d=x_d \right) = 2 F_{1,\b + \frac{d-1}{2}}(h)
	\]
	for $h \in [0,1]$. Using these observations, Equation \eqref{e7}, the affine Blaschke-Petkantschin formula, \Cref{thm:ABP}, and exploiting the rotational symmetry of the density, we get
	\begin{align*}
		\EE& T_{a,b}^{d,d-1}\left( Q_{n,d}^\b \right) =\ind_{\{a=0\}} \binom{n}{d-1} \frac{d \k_d}{2^{n-d} ((d-1)!)^b} \left(\frac{c_{d,\b}}{c_{d-1,\b}}\right)^{d-1} \\
		&\qquad\qquad \times \int\limits_{G(d,d-1)} \int\limits_{L^{d-1}} \nabla_{d-1}^{b+1}(x_1,\dots,x_{d-1}) \left( \prod_{i=1}^{d-1} f_{d-1,\b}(x_i) \right) \lambda_L^{d-1}\left( \dint (x_1,\dots,x_{d-1}) \right) \nu_{d-1}(\dint L) \\
		&\qquad + \binom{n}{d} \frac{d! \k_d}{2} \int\limits_{A(d,d-1)} \int\limits_{E^d} \PP\left( [X_1,\dots,X_d] \in \mathcal{F}_{d-1}\left(Q_{n,d}^\b\right) \Big| X_1=x_1,\dots,X_d=x_d \right) \\
		&\qquad\qquad \times \eta^a\left([x_1,\dots,x_d]\right) \Delta_{d-1}^{b+1}(x_1,\dots,x_d) \left( \prod_{i=1}^{d} f_{d,\b}(x_i) \right)  \lambda_d^E\left(\dint (x_1,\dots,x_d)\right) \mu_{d-1}(\dint E) \\
		&= \ind_{\{a=0\}} \binom{n}{d-1} \frac{d \k_d \EE_\b\left( \nabla_{d-1}^{b+1} \right)}{2^{n-d}((d-1)!)^b} \left(\frac{c_{d,\b}}{c_{d-1,\b}}\right)^{d-1}  \\
		&\qquad + \binom{n}{d} \frac{d! \k_d}{2} \int\limits_{0}^{1} h^a F_{1,\b+\frac{d-1}{2}}(h)^{n-d} \int\limits_{E^d} \Delta_{d-1}^{b+1}(x_1,\dots,x_d) \left( \prod_{i=1}^{d} f_{d,\b}(x_i) \right)  \lambda_d^E\left(\dint (x_1,\dots,x_d)\right) \dint h \\
		&= D_{n,d}^{\b,a,b} + C_{n,d}^{\b,b} \int\limits_{0}^{1} h^a f_{1,d\b + \frac{d-1}{2}(d+b+1)}(h) F_{1,\b+\frac{d-1}{2}}(h)^{n-d} \,\dint h,
	\end{align*}
	for all $a,b\in\RR$. For the last equation we followed again along the lines of the proof for the polytope $P_{n,d}^\b$.
	
	One can do the analogous computations for the beta$^\prime$-distribution. In this case one has to pay attention to the different range of integration, probability contents provided in \Cref{lem2} and transformation provided in \Cref{lem3}.\hfill $\Box$

	\begin{proof}[Proof of \Cref{lem:Miles}]
		Let $X_1,\dots,X_d$ be i.i.d.\ random points distributed according to a beta-distri\-bution on $\BB^d$ with parameter $\beta>-1$. From \eqref{e7} and the well known base-times-height-formula for the volume of simplices we get
		\begin{align*}
		\EE_\b\left( \nabla_d^\k(X_1,\dots,X_d) \right) &= (d!)^\k \EE_\b\left( \Delta_d^\k(0,X_1,\dots,X_d) \right) \\
		&= (d!)^\k \EE_\b\left( d^{-\k} \mathrm{dist}^\k(0,\mathrm{aff}(X_1,\dots,X_d)) \Delta_{d-1}^\k(X_1,\dots,X_d) \right)\\
		&=((d-1)!)^\k \EE_\b\left( \mathrm{dist}^\k(0,\mathrm{aff}(X_1,\dots,X_d)) \Delta_{d-1}^\k(X_1,\dots,X_d) \right).
		\end{align*}
		Rewriting this as an integral over $(\RR^d)^d$ and applying the affine Blaschke-Petkantschin formula yields
		\begin{align*}
			&\EE_\b\left( \nabla_d^\k(X_1,\dots,X_d) \right) \\
			&= ((d-1)!)^\k \int\limits_{(\RR^d)^d} \mathrm{dist}^\k(0,\mathrm{aff}(x_1,\dots,x_d)) \Delta_{d-1}^\k(x_1,\dots,x_d) \left( \prod_{i=1}^{d} f_{d,\b}(x_i) \right) \lambda_d^d(\dint (x_1,\dots,x_d)) \\
			&=b_{d,d-1} ((d-1)!)^{\k+1} \int\limits_{\SS^{d-1}} \int\limits_0^1 \int\limits_{E^d} h^\k \Delta_{d-1}^{\k+1}(x_1,\dots,x_d) \left( \prod_{i=1}^{d} f_{d,\b}(x_i) \right) \lambda_E^d(\dint (x_1,\dots,x_d)) \dint h~\sigma_{d-1}(\dint u) \\
			&=2b_{d,d-1} ((d-1)!)^{\k+1} \left( \frac{c_{d,\b}}{c_{d-1,\b}} \right)^d \EE_\b\left( \Delta_{d-1}^{\k+1}(X_1,\dots,X_d) \right) \int\limits_0^1 h^\k (1-h^2)^{d\b+\frac{d-1}{2}(d+\k+1)} \dint h \\
			&=b_{d,d-1} ((d-1)!)^{\k+1} \left( \frac{c_{d,\b}}{c_{d-1,\b}} \right)^d \EE_\b\left( \Delta_{d-1}^{\k+1}(X_1,\dots,X_d) \right) B\left(\frac{\k+1}{2},d\b+\frac{d-1}{2}(d+\k+1)+1\right),
		\end{align*}
		where the second to last equality follows from Lemma~\ref{lem3}. Here $B(\cdot,\cdot)$ stands for the beta function. Thus, rearranging and adjusting for the correct indices gives
		\begin{align*}
			\EE_\b\left( \Delta_{d}^{\k}(X_0,\dots,X_d) \right) =  \left( \frac{c_{d,\b}}{c_{d+1,\b}} \right)^{d+1} \frac{\EE_\b\left( \nabla_{d+1}^{\k-1}(X_0,\dots,X_d) \right)}{b_{d+1,d}(d!)^\k B\left( \frac{k}{2}, (d+1)\b+\frac{d}{2}(d+\k+1)+1 \right)},
		\end{align*}
		from which the claim follows by working out the constants. One sees from the last equality that what we have shown is only true for $\k>0$. However, since this is true for all $\k \in (0,\infty)$, it follows with a similar argumentation as in the proof of \Cref{t2} that, by analytic continuation, this also holds for all $\k>-1$. The corresponding result for the beta$^\prime$-distribution can be shown analogously.
	\end{proof}

\section{Expected volumes and intrinsic volumes: Proof of Theorem~\ref{t2}}\label{sec:proof_exp_volumes}

\begin{proof}[Proof of \Cref{t2}]
	We start by investigating the case of a beta-distribution with parameter $\b>-1$. Let first $\beta>-\frac 12$.
 	For an arbitrary linear hyperplane $L \in G(d+1,d)$, \Cref{lem1} implies
	\begin{equation}\label{eq:proj_P_n_beta}
		P_{n,d+1}^{\beta-\frac 12} |L \overset{d}{=} P_{n,d}^{\beta},
	\end{equation}
	where $\overset{d}{=}$ indicated equality in distribution.
By the Kubota formula stated in \Cref{thm:Kub}, we have
$$
V_d(P_{n,d+1}^{\beta-\frac 12}) = (d+1) \frac{\k_{d+1}}{2 \k_d} \int\limits_{G(d+1,d)} \Vol_d(P_{n,d+1}^{\beta-\frac 12} |L) \,\nu_k(\dint L).
$$
Taking the expectation, using Fubini's theorem to change the order of integration, and applying~\eqref{eq:proj_P_n_beta}, we obtain
\begin{align*}
\EE V_d(P_{n,d+1}^{\beta-\frac 12})
&= (d+1) \frac{\k_{d+1}}{2 \k_d} \int\limits_{G(d+1,d)} \EE \Vol_d(P_{n,d+1}^{\beta-\frac 12} |L)\, \nu_d(\dint L)\\
&=
(d+1) \frac{\k_{d+1}}{2\k_d} \int\limits_{G(d+1,d)} \EE \Vol_d(P_{n,d}^{\beta}) \,\nu_d(\dint L)\\
&=
(d+1) \frac{\k_{d+1}}{2\k_d} \EE\Vol_d(P_{n,d}^{\beta}).
\end{align*}
 Since the $d$th intrinsic volume $V_{d}$ of a $(d+1)$-dimensional polytope is half its surface area, we can write the above in terms of the $T$-functional with $a=0$ and $b=1$ as follows:
\begin{equation}\label{e2}
\EE \Vol_d(P_{n,d}^{\beta})
=
\frac{2\k_d}{(d+1)\k_{d+1}} \EE V_d(P_{n,d+1}^{\beta-\frac 12})
=
\frac{\k_d}{(d+1)\k_{d+1}} \EE T_{0,1}^{d+1,d}(P_{n,d+1}^{\beta-\frac 12}).
\end{equation}
Using Theorem \Cref{t1}, we obtain
\begin{equation}\label{eq:E_Vol_d_P_n_beta}
\EE \Vol_d(P_{n,d}^{\beta})
=
A_{n,d}^\b \int\limits_{-1}^{1}  \left(1-h^2\right)^{(d+1)\left(\b-\frac 12\right) + \frac{d}{2}(d+3)} F_{1,\b+\frac{d-1}{2}}(h)^{n-d-1} \,\dint h,
\end{equation}
where
$$
A_{n,d}^\b = \frac{\k_d}{(d+1)\k_{d+1}}
C_{n,d+1}^{\b-\frac 12,1}
=
\frac {(d+1) \kappa_d}{2^d \pi^{\frac{d+1}{2}}} \binom {n}{d+1} \left(\beta + \frac {d+1}2\right)  \left(\frac{\Gamma\left(\frac{d+2}{2} + \b \right)}{\Gamma\left(\frac {d+3} 2 + \b\right)} \right)^{d+1}.
$$
In order to derive the last formula we used elementary transformations involving  \Cref{lem:Miles} and the Legendre duplication formula for the gamma function. So far, we established formula~\eqref{eq:E_Vol_d_P_n_beta} for $\b > -\frac 12$ only because the proof was based on representation~\eqref{eq:proj_P_n_beta}. In order to prove that~\eqref{eq:E_Vol_d_P_n_beta} holds in the full range $\b > -1$, we argue by analytic continuation. First of all, the function $\beta \mapsto \EE \Vol_d(P_{n,d}^{\beta})$ is real analytic in $\b>-1$ as one can see from the representation
$$
\EE \Vol_d(P_{n,d}^{\beta}) = \int\limits_{(\BB^d)^n} \Vol_{d}([x_1,\ldots,x_n]) \left(\prod_{i=1}^n f_{d,\beta}(x_i)\right) \lambda_d^n\left( \dint \left( x_1,\dots,x_n \right)\right).
$$
Secondly, the function on the right-hand side of~\eqref{eq:E_Vol_d_P_n_beta} is also real analytic in $\b>-1$. Since these functions coincide for $\b>-\frac 12$, they must coincide in the full range $\b>-1$.

			
Similarly, we obtain in the case of a beta$^\prime$-distribution with parameter $\b$ the intrinsic volume $V_d$ of $\tilde{P}_{n,d}^{\beta}$ as
	\begin{equation}\label{e5}
		\EE V_d\left(\tilde{P}_{n,d}^{\beta}\right) = \frac{\k_d}{(d+1)\k_{d+1}} \EE T^{d+1,d}_{0,1} \left( \tilde{P}_{n,d+1}^{\beta+\frac{1}{2}} \right) \mbox{,}
	\end{equation}
	from which we get
	\begin{equation}\label{e6}
	\EE V_d\left(\tilde{P}_{n,d}^{\beta}\right) = \tilde{A}_{n,d}^\b \int\limits_{-\infty}^{\infty}  \left(1+h^2\right)^{-(d+1)\left(\b+\frac 12\right) + \frac{d}{2}(d+3)} \tilde{F}_{1,\b-\frac{d-1}{2}}(h)^{n-d-1} \,\dint h \mbox{,}
	\end{equation}
	for all $k=1,\dots,d$, with
	$$
	\tilde A_{n,d}^\b= \frac{\k_d}{(d+1)\k_{d+1}}
	\tilde{C}_{n,d+1}^{\b+\frac 12,1}
	=
	\frac {(d+1) \kappa_d}{2^d \pi^{\frac{d+1}{2}}} \binom {n}{d+1} \left(\beta - \frac {d+1}2\right)  \left(\frac{\Gamma\left(\b - \frac{d+1}{2} \right)}{\Gamma\left(\b - \frac {d} 2 \right)} \right)^{d+1}.
	$$
	Inserting the expressions from \Cref{t1} into \eqref{e2}--\eqref{e6} yields the statement of the theorem for $P_{n,d}^\b$ and $\tilde{P}_{n,d}^\b$. To obtain the results for $S_{n,d}^\b$, $Q_{n,d}^\b$, $\tilde{S}_{n,d}^\b$ and $\tilde{Q}_{n,d}^\b$ we only need to replace the corresponding constants and indices from \Cref{t1}, with the ones obtained here.
\end{proof}

\begin{proof}[Proof of \Cref{prop:expected_intrinsic}]
The idea is to represent the expected intrinsic volume as the expected volume of the random projection by means of Kubota's formula.
For every linear subspace $L\in G(d,k)$, \Cref{lem1} yields the
representation
$$
P_{n,d}^{\beta} |L  \overset{d}{=} P_{n,k}^{\beta + \frac{d-k}{2}}.
$$
Using Kubota's formula, see \Cref{thm:Kub}, in conjunction with Fubini's theorem and the above representation, we get
\begin{align*}
\EE V_k(P_{n,d}^{\beta}) &=
\binom{d}{k} \frac{\k_d}{\k_k \k_{d-k}} \int\limits_{G(d,k)} \EE  \Vol_k\left(P_{n,d}^\beta|L\right) \nu_k(\dint L)
\\
&=\binom{d}{k} \frac{\k_d}{\k_k \k_{d-k}} \int\limits_{G(d,k)} \EE \Vol_k\left( P_{n,k}^{\beta + \frac{d-k}{2}}\right) \nu_k(\dint L)\\
&=
\binom{d}{k} \frac{\k_d}{\k_k \k_{d-k}} \EE \Vol_k\left( P_{n,k}^{\beta + \frac{d-k}{2}}\right).
\end{align*}
Analogously, by \Cref{lem1}, i.e., the representation
$$
\tilde{P}_{n,d}^{\beta} |L  \overset{d}{=} \tilde{P}_{n,k}^{\beta - \frac{d-k}{2}}
$$
for every linear subspace $L \in G(d,k)$, and the same arguments as before, we have
$$
\EE V_k(\tilde{P}_{n,d}^{\beta}) = \binom{d}{k} \frac{\k_d}{\k_k \k_{d-k}} \EE \Vol_k\left( \tilde{P}_{n,k}^{\beta - \frac{d-k}{2}}\right).
$$
The corresponding results for $S_{n,d}^\b$, $Q_{n,d}^\b$, $\tilde{S}_{n,d}^\b$ and $\tilde{Q}_{n,d}^\b$ hold with the same argumentation.
\end{proof}

\begin{proof}[Proof of Proposition~\ref{prop:uniform_sphere}]
Let $\mu_{d,\b}$ be the beta distribution on $\BB^d$ with parameter $\b > -1$. Also, let $\mu_{d}$ be the uniform distribution on the sphere $\SS^{d-1}$. We show that $\mu_{d,\b}$ weakly converges to $\mu_d$, as $\beta \downarrow -1$. Let $(\b_k)_{k\in\NN}$ be any sequence converging to $-1$ from above. Since the set of probability measures on the unit ball $\BB^d$ is compact in the weak topology, there is a weak accumulation point $\nu$ of the sequence $(\mu_{d,\b_k})_{k\in\NN}$. We need to show that $\nu= \mu_d$. First of all, since $\lim_{\b \downarrow 0} \Gamma(\beta + 1) = +\infty$, it is clear that the density $f_{d,\b}(x)$ converges to $0$ uniformly in $x\in K$, for every compact subset $K$ of the open unit ball. It follows that $\mu_{d, \b_k}(K)$ converges to $0$ as $k\to\infty$, hence the probability measure $\nu$ is concentrated on the sphere $\SS^{d-1}$. Since the measures $\mu_{d, \b_k}$ are invariant under arbitrary orthogonal transformations of $\RR^d$, the same is true for the weak limit $\nu$. It follows that $\nu$ is a rotationally invariant probability measure on the sphere $\SS^{d-1}$, hence $\nu = \mu_d$.

We showed that $\mu_{d,\beta}\to \mu_d$ weakly, as $\beta \downarrow -1$. Hence, we have the weak convergence of product measures: $\mu_{d,\b}^{\otimes n} \to \mu_d^{\otimes n}$, as $\beta \downarrow -1$, for each $n\in\NN$. Since the functional $(x_1,\ldots,x_n) \mapsto \Vol_d([x_1,\ldots,x_n])$ is continuous on $\BB^d\times \dots \times \BB^d$, the continuous mapping theorem implies that the random variable $\Vol_d (P^{\b}_{n,d})$ converges in distribution to $\Vol_d (P_{n,d})$, as $\b \downarrow -1$. Since these random variables are bounded by $\kappa_d$, we obtain
$$
\lim_{\b\downarrow -1} \EE \Vol_d (P^{\b}_{n,d}) = \EE \Vol_d (P_{n,d}).
$$
It remains to observe that the expression for $\EE \Vol_d (P^{\b}_{n,d})$ given in \Cref{t2} converges to the expression for $\EE \Vol_d (P_{n,d})$ given in \Cref{prop:uniform_sphere}. This is a consequence of the dominated convergence theorem, the majorant being the function $(1-h^2)^{(d^2-3)/2}$. For symmetric convex hulls, the proof is similar.
\end{proof}

\appendix

\section{Particular values for small dimensions and a small number of points}

We collect some particular mean values for the random polytopes $P_{n,d}^\beta$ and $S_{n,d}^\beta$ for $d=2$, $d=3$ and with $\beta=0$ (uniform distribution in the unit ball) and $\beta=-1$ (uniform distribution on the sphere).

\begin{table}[H]
\begin{center}
\begin{tabular}{|c||c|c|c|c|}
\hline
\parbox[0pt][2em][c]{0cm}{}& $\EE\Vol_2(P_{n,2}^0)$ & $\EE\Vol_2(S_{n,2}^0)$ & $\EE S_1(P_{n,2}^0)$ & $\EE S_1(S_{n,2}^0)$ \\
\hline
\hline
\parbox[0pt][2em][c]{0cm}{}$n=3$ & ${35\over48\pi}$ & $35\over 12\pi$ & $128\over 15\pi$ & ${512\over 15\pi}-{104704\over 1575\pi^2}$ \\
\hline
\parbox[0pt][2em][c]{0cm}{}$n=4$ & ${35\over 24\pi}$ & ${35\over 6\pi}-{2816\over 135\pi^3}$ & ${256\over 15\pi}-{11075584\over 165375\pi^3}$ & ${1024\over 15\pi}-{88604672\over 165375\pi^3}$ \\
\hline
\parbox[0pt][2em][c]{0cm}{}$n=5$ & ${175\over 72\pi}-{23023\over 6912\pi^3}$ & ${175\over 18\pi}-{23023\over 432\pi^3}$ & ${256\over 9\pi}-{5537792\over 33075\pi^3}$ & ${1024\over 9\pi}-{88604672\over 33075\pi^3}+{204130238464\over 38201625\pi^4}$\\
\hline
\end{tabular}
\caption{Mean area and perimeter length of a random polygon and a symmetric random polygon generated by $n$ points uniformly distributed in the unit disc.}
\end{center}
\end{table}

\vspace{-0.5cm}

\begin{table}[H]
\begin{center}
\begin{tabular}{|c||c|c|c|c|}
\hline
\parbox[0pt][2em][c]{0cm}{} & $\EE\Vol_2(P_{n,2})$ & $\EE\Vol_2(S_{n,2})$ & $\EE S_1(P_{n,2})$ & $\EE S_1(S_{n,2})$\\
\hline
\hline
\parbox[0pt][2em][c]{0cm}{} $n=3$ & $3\over 2\pi$ & $6\over\pi$ & $12\over \pi$ & ${48\over \pi}-{96\over \pi^2}$\\
\hline
\parbox[0pt][2em][c]{0cm}{} $n=4$ & $3\over\pi$ & ${12\over \pi}-{48\over\pi^3}$ & ${24\over \pi}-{96\over \pi^3}$ & ${96\over \pi}-{768\over \pi^3}$\\
\hline
\parbox[0pt][2em][c]{0cm}{} $n=5$ & ${5\over\pi}-{15\over 2\pi^3}$ & ${20\over \pi}-{120\over \pi^3}$ & ${40\over \pi}-{240\over \pi^3}$ & ${160\over \pi}-{3840\over \pi^3}+{7680\over\pi^4}$\\
\hline
\end{tabular}
\caption{Mean area and perimeter length of a random polygon and a symmetric random polygon generated by $n$ points uniformly distributed on the unit circle.}
\end{center}
\end{table}

\vspace{-0.5cm}

\begin{table}[H]
\begin{center}
\begin{tabular}{|c||c|c|c|c|c|c|}
\hline
\parbox[0pt][2em][c]{0cm}{} & $\EE V_3(P_{n,3}^0)$ & $\EE V_3(S_{n,3}^0)$ & $\EE S_2(P_{n,3}^0)$ & $\EE S_2(S_{n,3}^0)$ & $\EE \mw_3(P_{n,3}^0)$ & $\EE \mw_3(S_{n,3}^0)$\\
\hline
\hline
\parbox[0pt][2em][c]{0cm}{} $n=4$ & $12\pi\over 715$ & $96\pi\over 715$ & $36\pi\over 77$ & $135\pi \over 112$ & $666\over 715$ & $6408\over 5005$\\
\hline
\parbox[0pt][2em][c]{0cm}{} $n=5$ & $6\pi \over 143$ & $195\pi\over 1024$ & $11448\pi \over 17017$ & $24048\pi \over 17017 $ & $1044\over 1001$ & $2421\over 1792$\\
\hline
\parbox[0pt][2em][c]{0cm}{} $n=6$ & $2070\pi \over 29393$ & $77472\pi \over 323323$ & $1314\pi\over 1547$ & $5661\pi\over 3584$ & $33102\over 29393$ & $454140\over 323323$\\
\hline
\end{tabular}
\caption{Mean volume, surface area and mean width of a random polytope and a symmetric random polytope generated by $n$ points uniformly distributed in the $3$-dimensional unit ball.}
\end{center}
\end{table}

\vspace{-0.5cm}

\begin{table}[H]
\begin{center}
\begin{tabular}{|c||c|c|c|c|c|c|}
\hline
\parbox[0pt][2em][c]{0cm}{} & $\EE V_3(P_{n,3})$ & $\EE V_3(S_{n,3})$ & $\EE S_2(P_{n,3})$ & $\EE S_2(S_{n,3})$ & $\EE \mw_3(P_{n,3})$ & $\EE \mw_3(S_{n,3})$\\
\hline
\hline
\parbox[0pt][2em][c]{0cm}{} $n=4$ & $4\pi\over 105$ & $32\pi\over 105$ & $4\pi\over 5$ & $2\pi$ & $6\over 5$ & $8\over 5$\\
\hline
\parbox[0pt][2em][c]{0cm}{} $n=5$ & $2\pi \over 21$ & $5\pi \over 12$ & $8\pi \over 7$ & $16\pi \over 7$ & $4\over 3$ & $5\over 3$\\
\hline
\parbox[0pt][2em][c]{0cm}{} $n=6$ & $10\pi \over 63$ & $32\pi \over 63$ & $10\pi\over 7$ & $5\pi \over 2$ & $10\over 7$ & $12\over 7$\\
\hline
\end{tabular}
\caption{Mean volume, surface area and mean width of a random polytope and a symmetric random polytope generated by $n$ points uniformly distributed on the $3$-dimensional unit sphere.}
\end{center}
\end{table}

\vspace{-0.5cm}

\begin{table}[H]
\begin{center}
\begin{tabular}{|c||c|c||c|c||c|c|}
\hline
\parbox[0pt][2em][c]{0cm}{} & $\EE{\bf f}_1(P_{n,2}^0)$ & $\EE{\bf f}_1(S_{n,2}^0)$ & $\EE{\bf f}_2(P_{n,3}^0)$ & $\EE{\bf f}_2(S_{n,3}^0)$ & ${\bf f}_2(P_{n,3})$ & ${\bf f}_2(S_{n,3})$\\
\hline
\hline
\parbox[0pt][2em][c]{0cm}{} $n=3$ & $3$ & $6-{32\over 3\pi^2}$ & -- & 8 & -- & 8\\
\hline
\parbox[0pt][2em][c]{0cm}{} $n=4$ & $4-{35\over 12\pi^2}$ & $8-{70\over 3\pi^2}$ & $4$ & $357\over 32$ & $4$ & $12$\\
\hline
\parbox[0pt][2em][c]{0cm}{} $n=5$ & $5 - {175\over 24\pi^2}$ & $10 -{175\over 3\pi^2}+ {5632\over 27\pi^4}$ & $840\over 143$ & $2000\over 143$ & $6$ & $16$\\
\hline
\parbox[0pt][2em][c]{0cm}{} $n=6$ & $6-{175\over 12\pi^2} + {23023\over 1152\pi^4}$ & $12-{350\over 3\pi^2}+{23023\over 36\pi^4}$ & $1090\over 143$ & $8485\over 512$ & $8$ & $20$\\
\hline
\end{tabular}
\caption{Mean number of edges and facets of a random polytope and a symmetric random polytope generated by $n$ points uniformly distributed in the unit ball. The last two columns collect the a.s.\ number of facets of a random polytope and a symmetric random polytope generated by $n$ random points uniformly distributed on the $3$-dimensional unit sphere.}
\end{center}
\end{table}

\vspace{-0.5cm}

\begin{table}[H]
\begin{center}
\begin{tabular}{|c||c|c|c|}
\hline
\parbox[0pt][2em][c]{0cm}{} & $\EE {\bf f}_1(\hat P_{n,2}^0)$ & $\EE {\bf f}_2(\hat P_{n,3}^0)$ & $\EE {\bf f}_3(\hat P_{n,4}^0)$\\
\hline
\hline
\parbox[0pt][2em][c]{0cm}{} $n=3$ & $3$ & -- & --\\
\hline
\parbox[0pt][2em][c]{0cm}{} $n=4$ & $6-{24\over \pi^2}$ & $4$ & --\\
\hline
\parbox[0pt][2em][c]{0cm}{} $n=5$ & $10-{60\over \pi^2}$ & ${20\over 3}-{10\over \pi^2}$ & $5$\\
\hline
\parbox[0pt][2em][c]{0cm}{} $n=6$ &  $15-{180\over \pi^2}+{720\over \pi^4}$ & $10-{30\over \pi^2}$ & $15-{200\over 3\pi^2}$\\
\hline
\parbox[0pt][2em][c]{0cm}{} $n=7$ &  $21-{420\over \pi^2}+{2520\over \pi^4}$ & $14-{70\over \pi^2}+{105\over \pi^4}$ & $35-{700\over 3\pi^2}$\\
\hline
\parbox[0pt][2em][c]{0cm}{} $n=8$ &  $28-{840\over \pi^2}+{10080\over \pi^4}-{40320\over \pi^6}$ & ${56\over 3}-{140\over \pi^2}+{420\over \pi^4}$ & $70-{2800\over 3\pi^2}+{101920\over 27\pi^4}$\\
\hline
\end{tabular}
\caption{The mean number of (spherical) facets of a random polytope generated by $n$ points uniformly distributed on the $2$-/$3$-/$4$-dimensional upper half-sphere.}
\end{center}
\end{table}

\subsection*{Acknowledgement}
We would like to thank Matthias Reitzner (Osnabr\"uck) for making reference \cite{BuchtaEllipsoid} available to us. We also thank the referee for his/her comments and suggestions.

DT was supported by the Deutsche Forschungsgemeinschaft (DFG) via RTG 2131 {\it High-Dimen\-sional Phenomena in Probability -- Fluctuations and Discontinuity}. ZK and CT were supported by the DFG Scientific Network {\it Cumulants, Concentration and Superconcentration}.


\vspace{1cm}

\footnotesize

\textsc{Zakhar Kabluchko:} Institut f\"ur Mathematische Stochastik, Westf\"alische Wilhelms-Universit\"at M\"unster\\
\textit{E-mail}: \texttt{zakhar.kabluchko@uni-muenster.de}

\bigskip

\textsc{Daniel Temesvari:} Fakult\"at f\"ur Mathematik, Ruhr-Universit\"at Bochum\\
\textit{E-mail}: \texttt{daniel.temesvari@rub.de}

\bigskip

\textsc{Christoph Th\"ale:} Fakult\"at f\"ur Mathematik, Ruhr-Universit\"at Bochum\\
\textit{E-mail}: \texttt{christoph.thaele@rub.de}

\end{document}